\documentclass[11pt,a4paper,reqno]{amsart}

\usepackage[utf8]{inputenc}

\usepackage{array,amsbsy,amscd,amsfonts,color,amssymb,amstext,amsmath,latexsym,amsthm,amscd
,multicol,graphicx,tikz}
\usepackage[english]{babel}
\usepackage{hyperref}

\numberwithin{equation}{section}

\makeindex

\newtheorem{theorem}{Theorem}
\newtheorem{proposition}[theorem]{Proposition}
\newtheorem{lemma}[theorem]{Lemma}
\newtheorem{corollary}[theorem]{Corollary}

\theoremstyle{definition}
\theoremstyle{definition}\newtheorem{definition}[theorem]{Definition}
\theoremstyle{definition}\newtheorem{remark}[theorem]{Remark}
\theoremstyle{definition}
\theoremstyle{definition}
\theoremstyle{definition}
\theoremstyle{definition}

\numberwithin{theorem}{section}

\def\proofof [#1] {\noindent {\bf Proof of #1. } }

\def\v #1.{\mathord{\raise 3pt\hbox{\mathsurround=0pt $\mathop\vee\limits^{#1}$\mathsurround=5pt}}}

\renewcommand{\AA}{\mathfrak{A}}
\newcommand{\BB}{\mathfrak{B}}

\newcommand{\FF}{\mathfrak{F}}
\newcommand{\A}{\mathcal{A}}
\newcommand{\B}{\mathcal{B}}
\newcommand{\D}{\mathcal{D}}
\newcommand{\I}{\mathcal{I}}
\newcommand{\K}{\mathcal{K}}
\renewcommand{\H}{\mathcal{H}}

\newcommand{\F}{\mathcal{F}}
\newcommand{\G}{\mathcal{G}}
\renewcommand{\L}{\mathcal{L}}
\newcommand{\N}{\mathbb{N}}
\newcommand{\NN}{\mathbb{N}_0}
\newcommand{\R}{\mathbb{R}}
\newcommand{\C}{\mathbb{C}}

\newcommand{\T}{\mathcal{T}}
\renewcommand{\S}{\mathcal{S}}

\newcommand{\unit}{\mathbf{1}}

\newcommand{\bp}{\begin{proof}}
\newcommand{\ep}{\end{proof}}
\newcommand{\bdp}{\begin{dproof}}
\newcommand{\edp}{\end{dproof}}
\newcommand{\ra}{\rightarrow}

\newcommand{\Uone}{\operatorname{U}(1)}

\newcommand{\CAR}{\operatorname{CAR}}

\newcommand{\Diff}{\operatorname{Diff}(S^1)}

\newcommand{\Mat}{\operatorname{M}}

\newcommand{\locn}{\operatorname{ln}}

\newcommand{\SVir}{\operatorname{SVir}}

\newcommand{\Vir}{\operatorname{Vir}}
\newcommand{\tr}{\operatorname{tr }}
\newcommand{\Ad}{\operatorname{Ad }}

\newcommand{\rmd}{\operatorname{d}}
\newcommand{\rmi}{\operatorname{i}}

\newcommand{\geo}{\operatorname{geo}}

\newcommand{\salg}{\ ^*\operatorname{-alg}}

\newcommand{\rme}{\operatorname{e}}

\newcommand{\eps}{\varepsilon}

\newcommand{\sgn}{\operatorname{sgn }}

\newcommand{\dom}{\operatorname{dom }}
\newcommand{\im}{\operatorname{im }}

\renewcommand{\ker}{\operatorname{ker }}

\newcommand{\spa}{\operatorname{span }}

\newcommand{\Cci}{C^{\infty}}

\newcommand{\id}{\operatorname{id}}

\newcommand{\ie}{{i.e.,\/}\ }

\newcommand{\eg}{{e.g.\/}\ }
\newcommand{\cf}{{cf.\/}\ }

\author{Robin Hillier} 
\thanks{Formerly Marie-Curie Fellow of the Istituto Nazionale di Alta Matematica, Roma, and supported by the ERC
Advanced Grant 227458 "Operator Algebras and Conformal Field Theory".}
\title{Super-KMS Functionals for \linebreak Graded-Local Conformal Nets}
\address{Department of Mathematics and Statistics, Lancaster University, Lancaster LA1 4YF, UK\\
E-mail: {\tt r.hillier@lancaster.ac.uk}}

\begin{document}

\begin{abstract}
Motivated by a few preceding papers and a question of R.~Longo, we introduce super-KMS functionals for graded translation-covariant nets over $\R$ with superderivations, roughly speaking as a certain supersymmetric modification of classical KMS states on translation-covariant nets over $\R$, fundamental objects in chiral algebraic quantum field theory. Although we are able to make a few statements concerning their general structure, most properties will be studied in the setting of specific graded-local (super-) conformal models. In particular, we provide a constructive existence and partial uniqueness proof of super-KMS functionals for the supersymmetric free field, for certain subnets, and for the super-Virasoro net with central charge $c\ge 3/2$. Moreover, as a separate result, we classify bounded super-KMS functionals for graded-local conformal nets over $S^1$ with respect to rotations.
\end{abstract}

\maketitle

\tableofcontents

\section{Introduction}

In the 1950s, Kubo, Martin and Schwinger studied certain thermodynamical equilibrium states of many-particle quantum systems \cite{MS}, enabling later on several physical applications, \eg in black hole dynamics or phase transitions. Haag, Hugenholtz and Winnink then used their construction to formulate an abstract algebraic-analytic relation which led to the notion of KMS states for arbitrary C*-dynamical systems. The probably most famous C*-dynamical systems with an immediate quantum field theoretical meaning are the C*-algebra of the canonical commutation (or anticommutation, respectively) relations on a given Hilbert $L^2$-space, in short CCR (or CAR), equipped with translation actions. Roughly speaking, they describe bosonic (or fermionic) free fields, and are therefore of fundamental interest in mathematical physics. A detailed discussion of the KMS condition for CCR and CAR algebras together with other specific situations can be found in \cite{BR2}. A more direct and specific study and classification of KMS states for the CCR and CAR algebras over a real infinite-dimensional Hilbert space with respect to a given symmetry group has been achieved in \cite{RST,RST2}. In this context it was shown that they are quasi-free, a property extensively investigated in \cite{Ara,Ara2,Ara1} for the CCR and CAR algebras.

Given the importance of KMS states to quantum field theoretical systems, it seems natural to study them also in the context of algebraic quantum field theory, the general axiomatic formulation of quantum field theories over arbitrary spacetimes by means of nets of quantum fields, \cf \cite{Haag} for an overview. The first steps in that direction were taken in \cite{BJ}. In later years quantum field theories with conformal symmetries on low-dimensional spacetimes became more and more interesting, due to their beautiful mathematical structure related to modular theory and subfactor theory of von Neumann algebras on the one hand, and to their physical meaning on the other hand, \cf \eg \cite{FG}. An extensive study of KMS states for chiral conformal nets over the real line $\R$ with respect to translations has been recently performed in \cite{CLTW1,CLTW2}, dealing in detail with several important models and providing a complete classification of KMS states for completely rational nets.

Supersymmetry is an important aspect of many quantum field theory models. Although its actual presence in nature is unclear, it nevertheless carries along a deep mathematical and physical structure. Combining supersymmetry with the above KMS property means in some sense to combine (thermo-) dynamics and supersymmetry into supersymmetric (thermo-) dynamics and corresponding supersymmetric phase transitions. Supersymmetric versions of conformal nets have been introduced mainly in \cite{CKL}, while a few relations to supersymmetric (or sometimes called ``graded") KMS functionals were studied in \cite{BL00,Lo01}. Apart from supersymmetric dynamics and general structural interest in those functionals, one of the main motivations and applications is the construction of certain entire cyclic cocycles \cite{JLW,Kas} and thus noncommutative geometric invariants \cite{Con94}. However, those constructions do only work for bounded super-KMS functionals, which cannot exist in the natural setting of superconformal nets over $\R$ \cite{BL00}. Furthermore, apart from this last negative statement, not much is known so far about the structure of such functionals for superconformal nets, we are even lacking examples so far. One first example of algebraic supersymmetry with unbounded but locally bounded super-KMS functionals and associated local-entire cyclic cocycle has been constructed for the supersymmetric free field \cite{BG}, which, however, still has to be put into the framework of superconformal nets of von Neumann algebras.

From this introduction, several questions seem to emerge quite naturally:
\begin{itemize}
\item[(1)] Can we provide a suitable definition of super-KMS functionals for graded translation-covariant nets over $\R$ of von Neumann algebras and determine abstract properties, extending the results in \cite{BL00}?
\item[(2)] Can we find examples of super-KMS functionals for some explicit (say, superconformal) nets on $\R$?
\item[(3)] Can we formulate a general classification of super-KMS functionals for graded translation-covariant nets, in analogy to \cite{CLTW1,CLTW2}?
\item[(4)] Do super-KMS functionals for graded translation-covariant nets give rise to generalized homotopy-invariant entire cyclic cocycles and geometric invariants of the original net, generalizing \cite{CHL}?
\end{itemize}

In the present paper, we shall deal with some of these questions. Regarding the first one, we start in Section \ref{sec:gen} with a definition of super-KMS functionals which seems suitable in general and in the context of graded translation-covariant nets over $\R$. This permits us to make a few general statements, \eg that a usual KMS functional on a completely rational conformal net has to be unbounded and nonpositive if it does not coincide with the unique KMS state, or that the Jordan decomposition of a super-KMS functional cannot be obtained as an inductive limit. We also outline general motivations and applications of super-KMS functionals.

As our original motivation comes from graded-local conformal or superconformal nets, we turn in Section \ref{sec:FF} to the second question, dealing with some important superconformal models: mainly the fermionic and the supersymmetric free field nets, briefly general conformal subnets of the fermionic free field net, and finally the super-Virasoro nets with central charges $c\ge 3/2$. We explicitly construct super-KMS functionals with respect to translations on the quasi-local C*-algebra of those nets and show uniqueness for some of them under certain additional regularity assumptions. As an essential ingredient and byproduct independent of the context of conformal nets, we find a sort of generalization of Araki's criterion on quasi-equivalence of quasi-free states on the CAR algebra \cite{Ara} to a well-behaved class of (nonpositive but bounded) quasi-free functionals on the CAR algebra, \cf Theorem \ref{prop:FF-comm2}.

As evident from the preceding discussion, question 3 will be impossible to answer. Concerning the last question, however, we can say quite a lot. Since this turns out to be a deep issue, too, we placed it in \cite{Hil2} and rephrase here just that the answer is a ``conditional yes".

Our final section deals with super-KMS functionals for graded-local conformal nets over $S^1$ with respect to the rotation group action. With a suitable and natural definition of the underlying global algebra for those nets, we shall see that there bounded super-KMS functionals do exist and they are actually all combinations of super-Gibbs functionals. Hence to them one can apply the constructions in \cite{JLW,Kas}, and in special cases this has in fact been exploited in the recent articles \cite{CHL} and \cite{CHKLX}.

In order to keep this paper as concise as possible, we have postponed a lot of material to the appendix. In Appendix \ref{sec:Araki} we prove the above-mentioned quasi-equivalence criterion, and in Appendix \ref{sec:proofs} we collect several lengthy and technical proofs of statements made in Section \ref{sec:FF}. 

One final important remark: In the present paper, we deal with nets over $\R$, whose physical meaning is that of a light-ray in two-dimensional Minkowski spacetime, especially interesting for so-called chiral theories (\cf \cite{Wei05} for a good overview). We restrict to this ``spacetime" setting since our results in Sections \ref{sec:gen} and \ref{sec:FF} treat that quite ample case. However, most of the initial definitions here make sense also for nets over more general spacetimes with a designated (fixed) ``translation direction".

\section{General aspects of super-KMS functionals}\label{sec:gen}

\subsection*{Preliminaries and notation}

Let $\I$ stand for the set of nonempty bounded open intervals in $\R$ and $\I_S$ the set of nonempty nondense open intervals of the unit circle $S^1$ (sometimes regarded as being embedded in $\C$). By means of the Cayley transformation, $\R$ may always be identified with the subset $S^1\setminus\{-1\}$ of $S^1$, and hence $\I$ with a subset of $\I_S$. We write
\[
\T^1:= \{z\in\C : 0\le \Im z \le 1 \}
\]
for the standard closed strip in the complex plane.

A \emph{graded translation-covariant net $\A$ over $\R$} is a map $I \mapsto \A(I)$ from the set $\I$ to the set of von Neumann algebras acting on a common infinite-dimensional separable Hilbert space $\H$ satisfying the 
following properties:
\begin{itemize}
\item[-] \emph{Isotony.} $\A(I_1)\subset \A(I_2)$ if $I_1,I_2\in\I$ and $I_1\subset I_2$.
\item[-] \emph{Grading.} There is a fixed selfadjoint unitary $\hat{\Gamma}\not=\unit$ (the grading unitary) on $\H$ satisfying 
$\hat{\Gamma} \A(I) \hat{\Gamma} = \A(I)$ for all $I\in \I$. We write $\gamma=\Ad\hat{\Gamma}$ and define the usual \emph{graded commutator}
\begin{align*}
[x,y]=&\; xy+\frac14 (y-\gamma(y))(x-\gamma(x))- \frac14 (y+\gamma(y))(x-\gamma(x))\\
&- \frac14 (y-\gamma(y))(x+\gamma(x))- \frac14 (y+\gamma(y))(x+\gamma(x)), \quad x,y\in\A(I).
\end{align*}
\item[-] \emph{Translation-covariance.} There is a strongly continuous unitary representation on $\H$ of the translation group $\R$ with infinitesimal generator $P$, commuting with $\hat{\Gamma}$ and inducing a one-parameter automorphism group $t\mapsto\alpha_t:=\Ad(\rme^{\rmi tP})$ on $B(\H)$ such that $\alpha_t$ restricts to *-isomorphisms from $\A(I)$ to $\A(t+I)$, for every $t\in\R$ and $I\in\I$; moreover $\alpha$ is asymptotically graded-abelian in the norm topology of the quasi-local C*-algebra $\AA$ introduced below in \eqref{eq:gen-qlocal}:
\[
\lim_{t\ra \infty} [x,\alpha_t(y)] = 0, \quad x,y\in \A(I), I\in\I.
\]
\item[-] \emph{Positivity of the energy.} $P$ is positive.
\end{itemize}

Although the general statements in the present section make use only of this general framework of nets, our original interest and all our studies in the subsequent sections deal with the special case of superconformal nets over $\R$ or $S^1$. To this end, let $\Diff$ be the group of orientation-preserving diffeomorphisms of $S^1$ and $\Diff^{(\infty)}$ its universal covering group; then recall from \cite{CKL,CHKL,CHL} that a \emph{graded-local conformal net $\A_S$ over $S^1$} is a map $I\in\I_S \mapsto \A_S(I)$ satisfying the above axioms but with $\I$ replaced by $\I_S$ and with the following additional properties:
\begin{itemize}
\item[-] \emph{Diffeomorphism-covariance.} There is a strongly continuous projective unitary representation $U:\Diff^{(\infty)}\ra B(\H)$ extending the unitary representation of the translation group and such that
\[
 U(g)\A_S(I)U(g)^* = \A_S(\dot{g}I), \quad g\in \Diff^{(\infty)}, I\in\I_S,
\]
and
\[
 U(g) x U(g)^* = x,\quad x\in \A_S(I'), g\in \Diff^{(\infty)}_I, I\in\I_S,
\]
where $\Diff^{(\infty)}_I$ denotes the $\unit$-connected component of $\{g\in\Diff^{(\infty)}: gz=z$ $\forall z\in S^1\setminus I\}$ and $I'$ the interior of $I^c$.
\item[-] \emph{Existence and uniqueness of the vacuum.} There exists a translation-invariant vector $\Omega\in\H$ which is unique up to a phase and cyclic for $\bigvee_{I\in\I_S} \A_S(I)$.
\item[-] \emph{Graded locality.} The grading unitary $\hat{\Gamma}$ satisfies additionally
\[
 \A_S(I_1)\subset Z \A_S(I_2)'Z^*,\quad I_1,I_2\in\I , \bar{I}_1\cap\bar{I}_2 = \emptyset,
\]
where
\[
 Z:=\frac{\unit - \rmi\hat{\Gamma}}{1-\rmi}.
\]
Notice that this is stronger than asymptotic graded-abelianess.
\end{itemize}
We shall treat super-KMS aspects of these nets in Section \ref{sec:S1}, but throughout the rest of the article they shall play no role. Moreover, if the grading is trivial, then the above definition becomes that of (diffeomorphism-covariant) \emph{local conformal nets} (\cf \eg \cite{Carpi,FG,Wei05}). 

Finally, a \emph{graded-local conformal net $\A$ over $\R$} is the net obtained by restricting $\A_S$ to $\I$, and covariance is now with respect to the stabilizer subgroup of $\Diff^{(\infty)}$ for the point $-1$. It forms a special case of a graded translation-covariant net. It is called \emph{superconformal} if it contains the super-Virasoro net $\A_{\SVir,c}$ (defined first in \cite[Sec.6]{CKL}, \cf also Appendix \ref{sec:proofs} below) as a diffeomorphism-covariant subnet with the same central charge $c$ and the projective representation $U$ of $\Diff^{(\infty)}$ making $\A$ diffeomorphism-covariant satisfies
\[
U(\Diff^{(\infty)}_I) \subset \A_{\SVir,c}(I) \subset \A(I), \quad I\in\I,
\]
\cf \cite[Def.2.11]{CHL}. There are several examples of such nets, \cf \cite[Sec.6]{CHL}, and we discuss some of them from the super-KMS point of view in Section \ref{sec:FF}.

The \emph{quasi-local C*-algebra} corresponding to a (graded) net $\A$ over $\R$ is defined as the C*-direct limit
\begin{equation}\label{eq:gen-qlocal}
\AA := \lim_{\ra} \A(I)
\end{equation}
over $I\in\I$, \cf also \cite{BR2,CLTW1,Haag}, noting that $\I$ is directed (in contrast to $\I_S$). For all $I\in\I$, $\A(I)$ is naturally identified with a subalgebra of $\AA$. When referring to particular models for $\A$, we shall say so explicitly. Throughout this paper we use gothic letters for the quasi-local C*-algebra of the nets with corresponding calligraphic letter.  

As in \cite[Sec.2]{CLTW1}, which shall serve as guideline for many aspects of our setting, we write $\alpha$ again for the induced one-parameter group of automorphisms of $\AA$, and we denote its infinitesimal generator by the derivation $(\delta_0,\dom(\delta_0))$, which is locally (\ie on every $\A(I)$) $\sigma$-weakly densely defined and formally given there by the commutator with $P$. Furthermore, we write  $\AA_\alpha\subset \AA$ for the *-subalgebra of analytic elements of $\alpha$: the elements $x\in\AA$ such that $t\in\R\mapsto \alpha_t(x)\in\AA$ extends to an entire analytic function, denoted $z\in\C\mapsto \alpha_z(x)$, \cf \cite[Sec.3.2]{BR2}. 

Let $\A$ be a graded translation-covariant net. A \emph{superderivation} on $\AA$ with respect to the grading $\gamma$ and translation group $\alpha$ is a linear map $\delta:\dom(\delta)\subset \AA\ra \AA$ such that:
\begin{itemize}
\item[$(i)$] $\dom(\delta)\subset \AA$ is an \emph{$\alpha$-$\gamma$-invariant} (\ie globally invariant under the action of every $\alpha_t$, $t\in\R$, as well as $\gamma$) unital *-subalgebra, with 
\[
\alpha_t\circ\delta(x)=\delta\circ\alpha_t(x), \quad \gamma\circ\delta(x)=-\delta\circ\gamma(x), \quad \delta(x^*)=\gamma(\delta(x)^*), \quad x\in\dom(\delta), t\in\R;
\]
\item[$(ii)$] $\delta(xy)= \delta(x)y+\gamma(x)\delta(y)$, for all $x,y\in\dom(\delta)$;
\item[$(iii)$] $\delta_I := \delta \restriction_{\dom(\delta)\cap\A(I)}$ is a ($\sigma$-weakly)-($\sigma$-weakly) closed $\sigma$-weakly densely defined map with image in $\A(I)$, for every $I\in\I$;
\item[$(iv)$] $\Cci(\delta_I) := \bigcap_{n\in\N} \dom(\delta^n_I) \subset \A(I)$ is $\sigma$-weakly dense, for every $I\in\I$.
\end{itemize}
By $\dom(\cdot)_I$ we always mean $\dom(\cdot)\cap \A(I)$ and thus $\dom(\delta_I)=\dom(\delta)_I$; $\dom(\cdot)_c$ stands for the union over $I\in\I$ of $\dom(\cdot)_I$, which in some cases may actually be equal to $\dom(\cdot)$. We then call $(\AA,\gamma,\alpha,\delta)$ a \emph{graded quantum dynamical system}. We shall be interested in modifications of the usual KMS condition on $(\AA,\alpha)$ \cite{CLTW1}, and for simplicity we consider only the case of inverse temperature $\beta=1$; the other cases of $\beta\not=0,\infty$ can be treated analogously. We shall deal with several examples of this setting in Section \ref{sec:FF}.

\noindent \textbf{Convention.} All *-algebras in this paper are understood to be unital and all Hilbert spaces separable if not stated otherwise.

\subsection*{Super-KMS functionals}

Super-KMS functionals are the core objects of this paper and we choose the following definition, which was motivated by the corresponding ones in \cite{BG,BL00} but is actually much stronger and more suitable for the theory and examples developed in the present paper and \cite{Hil2}.

\begin{definition}\label{def:gen-sKMSfunctional}
A \emph{super-KMS functional} (in short sKMS functional) $\phi$ on a graded quantum dynamical system $(\AA, \gamma, \alpha, \delta)$ is a linear functional defined on a  *-subalgebra $\dom(\phi)\subset \AA$ such that:
\begin{itemize}
\item[$(S_0)$] Domain properties: $\phi(x^*)=\overline{\phi(x)}$, for all $x\in\dom(\phi)$; $\dom(\phi)_I\subset \A(I)$ is $\sigma$-weakly dense, for all $I\in\I$, and $\dom(\phi)$ is globally $\alpha$-$\gamma$-invariant.
\item[$(S_1)$] Local normality: $\phi_I:=\phi\restriction_{\dom(\phi)\cap\A(I)}$ is bounded and extends to a normal (\ie $\sigma$-weakly continuous) linear functional on $\A(I)$, denoted again $\phi_I$, for all $I\in\I$.
\item[$(S_2)$] sKMS property: for every $x,y\in \dom(\phi)$, there is a continuous function $F_{x,y}$ on the strip $\T^1$ which is analytic on the interior, satisfying
\begin{equation}\label{eq:gen-S2a}
F_{x,y}(t) = \phi(x \alpha_t(y)),\quad F_{x,y}(t+\rmi) = \phi (\alpha_t(y)\gamma(x)), \quad t\in\R,
\end{equation}
and there are constants $C_0>0$ and $p_0\in2\N$ depending only on $x,y,\phi$ such that
\begin{equation}\label{eq:gen-S2b}
|F_{x,y}(t)| \le C_0 (1+|\Re(t)|)^{p_0}, \quad t\in\T^1.
\end{equation}
\item[$(S_3)$] Normalization: $\phi(\unit)=1$.
\item[$(S_4)$] Derivation invariance: $\phi\circ \delta =0$ on $\dom(\delta)_c$.
\item[$(S_5)$] Weak supersymmetry: for every $x,z\in\dom(\phi)_c$ and $y\in\dom(\delta^2)_c$, we have
\[
\phi(x\delta^2(y)z) = -\rmi\frac{\rmd}{ \rmd t} \phi(x\alpha_t(y)z)\restriction_{t=0}.
\]
\item[$(S_6)$] Local-exponential boundedness: there are constants $C_1,C_2>0$ such that\linebreak $\|\phi \restriction_{\dom(\phi)_I}\|\le C_1\rme^{C_2 |I|^2}$, for every $I\in\I$. Here and throughout this paper we write $|I|:= \sup \{ |x| : x\in I\}$.
\end{itemize}
Sometimes we will dispense with certain of these properties and still call the functional sKMS, but we shall say so explicitly whenever this is the case.

Let $\B$ be a local conformal net with quasi-local C*-algebra $\BB$ and one-parameter automorphism group $\alpha$ corresponding to translations. A \emph{KMS functional} $\phi$ on $(\BB,\alpha)$ is a (possibly unbounded) linear functional defined on a  *-subalgebra $\dom(\phi)\subset \BB$ satisfying $(S_0)$-$(S_3)$ with $\gamma=\id_\BB$ and possibly without \eqref{eq:gen-S2b}.
\end{definition}

\subsection*{Some general properties}

The theory of usual KMS states motivates the following proposition:

\begin{proposition}\label{prop:gen-alpha-inv}
Let $(\phi,\dom(\phi))$ be a functional on the graded quantum dynamical system $(\AA,\gamma,\alpha,\delta)$ with properties $(S_0,S_1,S_3)$. Let $\AA_{\alpha,\phi}\subset \AA_\alpha\cap\dom(\phi)$ be an $\alpha$-$\gamma$-invariant *-subalgebra such that, for every $x,y\in\AA_{\alpha,\phi}$ and $z\in\T^1$, we have $\alpha_z(y)\in\dom(\phi)$ and the function $z\in\T^1\mapsto \phi(x\alpha_z(y))$ is continuous on $\T^1$ and analytic on the interior of $\T^1$. Consider the following condition:
\begin{itemize}
\item[$(S_2')$] For all $x,y\in\AA_{\alpha,\phi}$, we have
\[
\phi(x \alpha_{\rmi}(y))= \phi(y \gamma(x)),
\]
and there are constants $C_0>0$ and $p_0\in 2\N$ such that
\[
|\phi(x\alpha_t(y))| \le C_0 (1+|\Re(t)|)^{p_0}, \quad t\in\T^1.
\]
\end{itemize}
Then 
\begin{itemize}
\item[$-$] $(S_2)$ implies $(S_2')$;
\item[$-$] $(S_2')$ implies $(S_2)$ for $x,y\in\AA_{\alpha,\phi}$ instead of $\dom(\phi)$. If moreover $(S_6)$ holds and $\AA_{\alpha,\phi}\cap \A(I)\subset \A(I)$ is $\sigma$-weakly dense, for all $I\in\I$, then $(S_2')$ implies \eqref{eq:gen-S2a} actually for every $x,y\in\bigcup_{I\in\I}\A(I)$ and hence in $\dom(\phi)_c$.
\end{itemize}
Furthermore, $(S_2)$ implies:
\begin{itemize}
\item[$(S_7)$] Translation invariance: $\phi\circ \alpha_t = \phi$ on $\dom(\phi)$, for all $t\in\R$;
\item[$(S_8)$] Gradedness: $\phi\circ\gamma=\phi$ on $\dom(\phi)$;
\item[$-$] if moreover $(S_6)$ holds then \eqref{eq:gen-S2a} holds for every $x,y\in\bigcup_{I\in\I}\A(I)$.
\end{itemize}
\end{proposition}

Let us postpone the main part of the proof to Appendix \ref{sec:proofs} but instead drop a few remarks here. Properties $(S_7)$ and $(S_8)$ are proved in \cite[Prop.5.3]{BG} and are in fact almost immediate  consequences of $(S_2)$, obtained basically by studying the functions $F_{x,\unit}$ and $F_{\unit,x}$ in $(S_2)$. It follows in particular from that proof that $(S_7)$ and $(S_8)$ hold on any given $\alpha$-$\gamma$-invariant *-subalgebra $A\subset \dom(\phi)$ instead of $\dom(\phi)$ if $(S_2)$ holds on $A$, a fact used in the proof of $(S_2)$ on $\dom(\phi)$. The analyticity condition \eqref{eq:gen-S2a} on $\bigcup_{I\in\I} \A(I)$ is crucial when restricting to subnets, which might have trivial intersection with the original $\dom(\phi)$. In general, \eqref{eq:gen-S2b} need not hold on $\bigcup_{I\in\I} \A(I)$ and thus neither on such a subnet, but it will be needed only on the original $\dom(\phi)$. It is well-known \cite[Prop.5.3.7]{BR2} that in the case of states and trivial grading, property $(S_2')$ is equivalent to $(S_2)$, while in the present setting apparently this is no longer true, owing to the unboundedness of $\phi$ and the fact 
that $\AA_{\alpha,\phi}\cap\A(I)$ may therefore be trivial.

\begin{proposition}\label{prop:gen-obstruction}
Let $(\A,\gamma,\alpha)$ be a graded translation-covariant net and $(\phi,\dom(\phi))$ a functional on $\AA$ satisfying $(S_0)$-$(S_3)$. Then:
\begin{itemize}
\item[$(1)$] $(\phi,\dom(\phi))$ is neither positive nor bounded.
\item[$(2)$] The functionals $|\phi_I|$ and $\phi_I^\pm:= \frac12(|\phi_I|\pm \phi_I)$ obtained through restriction are individually well-defined, bounded and positive, but they do not form a directed system with respect to restriction, so they do not give rise to positive (unbounded) functionals on $\AA$.
\end{itemize}
\end{proposition}

In particular, we have to point out that because of (2) the interesting construction of $|\phi|$ suggested in \cite[Sec.4]{Sto06} does not work here.

Before entering the proof, recall (\eg from \cite[Th.4.3.6]{KR2}) that every bounded functional $\phi$ on a C*-algebra has a unique Jordan decomposition $\phi=\phi^+ - \phi^-$ with $\phi^\pm$ mutually orthogonal positive functionals such that $\|\phi\|=\|\phi^+\|+\|\phi^-\|$.

\begin{proof}
(1) is given in \cite[Lem.2\&Cor.8]{BL00}.

(2) By assumption $(S_1)$, the local restrictions $\phi_I$ are bounded, so $|\phi_I|$ and $\phi_I^\pm$ are well-defined bounded positive functionals on $\A(I)$. Suppose $(\phi^\pm_I)_{I\in\I}$ forms a directed system, \ie for every inclusion $\bar{I}_1\subset I_2$ we have
$\phi^\pm_{I_2}\restriction_{\A(I_1)}=\phi^\pm_{I_1}$. Then
\begin{align*}
\|\phi_{I_1}\| 
=& \|\phi_{I_1}^+\| + \|\phi_{I_1}^-\|= \phi^+_{I_1}(\unit) +\phi^-_{I_1}(\unit) \\
=& \phi^+_{I_2}\restriction_{\A(I_1)}(\unit)+ \phi^-_{I_2}\restriction_{\A(I_1)}(\unit)
= \phi^+_{I_2}(\unit)+ \phi^-_{I_2}(\unit) \\
=& \|\phi_{I_2}^+\| + \|\phi_{I_2}^-\| = \|\phi_{I_2}\|,
\end{align*}
since $\unit\in \bigcap_{I\in\I}\A(I)$. But this would imply that $\|\phi_{I}\|$ does not depend on $I$, hence the boundedness of $\phi$ on $\AA$ as an inductive limit of functionals of the same norm -- a contradiction to part (1). So $\phi_I^\pm$ do not satisfy isotony, and neither do $|\phi_I|$, for $I\in\I$, and we cannot define unbounded positive functionals $\phi^\pm$ on $\AA$ as inductive limits such that $\phi=\phi^+-\phi^-$.
\end{proof}

Let us drop for a moment the grading assumption and look at completely rational local conformal nets over $\R$, \cf \cite{KLM}, which provide a large class of examples of conformal nets. Let $\B$ be such a net with quasi-local C*-algebra $\BB$. The KMS states on $\BB$ with respect to translations were determined in \cite{CLTW1}, and it turned out that in this case, the so-called \emph{geometric KMS state} (defined \eg in \cite[Sect.2.8]{CLTW1}) is the unique KMS state. One might ask what happens if one asks only for bounded KMS functionals, \ie functionals $\phi$ on $\BB$ satisfying $(S_0)$-$(S_3)$ with $\dom(\phi)=\BB$ and $\gamma=\id_\BB$; the answer is 

\begin{proposition}\label{prop:gen-rationalKMS}
Suppose $\B$ is a completely rational local conformal net over $\R$. Then every bounded KMS functional on its quasi-local C*-algebra $\BB$ with respect to translations is automatically positive and coincides with the geometric KMS state.
\end{proposition}

\begin{proof}
Let $\phi$ be a bounded KMS functional on $\BB$; then we have a well-defined decomposition $\phi=\phi^+-\phi^-$. As shown in \cite[Lem.2]{BL00}, $|\phi|=\phi^++\phi^-$ satisfies the KMS condition on $\BB$; the positivity of $|\phi|$ implies moreover that it is a multiple of the unique (geometric) KMS state on $\BB$, \cf \cite{CLTW1}. Since $\phi^\pm= \frac12(|\phi|\pm \phi)$ is the convex combination of two non-normalized KMS functionals, it has to be a (not necessarily normalized) KMS functional again; by construction it is bounded and positive although not necessarily faithful, so it has to be a multiple of the geometric KMS state, too, by the same reasoning as above. Thus $\phi$, which is normalized, coincides with a multiple of the geometric KMS state.
\end{proof}

\begin{remark}\label{rem:gen-S1-Moriya}
(1) Notice that the situation is completely different for completely rational graded-local conformal nets over the circle $S^1$ and sKMS functionals with respect to the periodic rotation action, like those treated in \cite{CHKL,CHL,CHKLX}. In that case, there are in fact bounded though nonpositive sKMS functionals on the universal C*-algebra of the net (in its universal locally normal representation) as explained in Section \ref{sec:S1}. 

(2) The nonpositivity of sKMS functionals on graded translation-covariant nets makes it difficult or even impossible to extend fundamental constructions on single algebras like those by Moriya \cite{Mor10,Mor11} to the present setting of nets. He considers \emph{supersymmetric states $\phi$ on C*-algebras}, \ie satisfying $\phi\circ\delta=0$, and in this case the GNS construction gives rise to supercharges for $\delta$. An extension to the setting of nets of von Neumann algebras should be natural, but compatibility with the sKMS property cause problems like those mentioned above (nonpositivity) with which we cannot deal here in further detail.

(3) A classification of sKMS functionals in general seems to be out of reach without further specifications and without a general theory of unbounded linear functionals on C*-algebras. Already in the non-graded special case where $\B$ is a completely rational local net over $\R$, the question of existence and uniqueness of \emph{unbounded nonpositive KMS functionals} on $\BB$ is not clear, whereas we know from Proposition \ref{prop:gen-rationalKMS} that there exists a unique \emph{bounded} KMS functional on $\BB$ namely the geometric KMS state. An example of an unbounded KMS functional for a completely rational local conformal net will be encountered in Corollary \ref{cor:FF-even-subnet} in the setting of the chiral Ising model, the even subnet of the free fermion net.
\end{remark}

\section{Model analysis of super-KMS functionals for superconformal nets over  $\R$}\label{sec:FF}

\subsection*{Fermions and quasi-free functionals}

Let $\K$ be a complex Hilbert space with an anti-unitary involution $\Gamma$. Then the \emph{selfdual CAR algebra} $\CAR(\K,\Gamma)$ is the  C*-algebra generated by elements $F(f)$, for all $f\in\K$, which are linear in $f$ and satisfy the canonical anticommutation relations
\begin{equation}\label{eq:FF-CAR}
 F(f)^* F(g)+ F(g) F(f)^* = \langle f,g\rangle ,\quad F(f)^*= F(\Gamma f), \quad f,g\in\K.
\end{equation}
These relations define, in particular, the C*-norm on $\CAR(\K,\Gamma)$. A \emph{quasi-free state} on $\CAR(\K,\Gamma)$ is a state which vanishes on all odd degree monomials in $F$ and on those of even degree it satisfies
\begin{equation}\label{eq:FF-qfree}
 \phi(F(f_1)\cdots F(f_{2n})) = \sum_{j=2}^{2n} (-1)^j \phi(F(f_1)F(f_j)) \cdot \phi \Big( \prod_{i\not=1,j} F(f_i) \Big).
\end{equation}
It is thus completely determined by its 2-point function, a certain sesquilinear map which in turn corresponds to a unique operator $S\in B(\K)$ such that $0\le S=S^*= \unit -S\le \unit$ and $\phi(F(f)^*F(g))= \langle f, S g\rangle$, \cf \cite{Ara} for further information. By a \emph{quasi-free functional} on $\CAR(\K,\Gamma)$ we shall mean a linear functional $\phi$ whose domain $\dom(\phi)\subset \CAR(\K,\Gamma)$ contains the *-algebra generated by monomials $F(f_1)\cdots F(f_n)$, with $f_i\in\K_0$ and $n\in\NN$ (where $\K_0\subset \K$ is a certain dense vector subspace); moreover, $\phi$ has to vanish on all odd degree monomials and its value on even degree monomials is defined as in \eqref{eq:FF-qfree} but with 2-point function
\[
\phi(F(f)^*F(g))= \theta(f,g), \quad f,g\in\K_0,
\]
where $\theta: \dom(\theta)\ra \C$ is a certain sesquilinear map such that $\K_0\times\K_0 \subset \dom(\theta)\subset\K\times\K$ is dense and $\theta(f,g)=\overline{\theta(g,f)}$. In case $\theta$ is bounded and $\dom(\theta)=\K\times\K$, there is clearly a corresponding operator $S\in B(\K)$ such that $S^*=S = \unit - \Gamma S\Gamma$ and $\theta=\langle\cdot , S\cdot \rangle$. The other way round, any such $S$ uniquely determines a bounded quasi-free functional on $\CAR(\K,\Gamma)$, which we denote by $\phi_S$; it is a state iff moreover $0\le S\le \unit$.
 
From now on, let $\K= L^2(\R,\C^d)$, with $d\in\N$, $\K_I=L^2(I,\C^d)\subset\K$ and $\Gamma: f\mapsto \bar{f}$ component-wise complex conjugation on $\K$ and $\K_I$, write $\K^\Gamma= L^2(\R,\R^d)$ and $\K_I^\Gamma=L^2(I,\R^d)$, and let $\S(\R,\C^d)$ and $\S(\R,\R^d)$ be the corresponding Schwartz spaces (of $\C^d$ or $\R^d$-valued rapidly decreasing smooth functions on $\R$), dense vector subspaces of $\K$ and $\K^\Gamma$, respectively. For $X=\C^d$ or $\R^d$, we write $\Cci_c(\R,X)$ for those smooth $X$-valued functions which have actually compact support, $\Cci_c(I,X)$ for those with compact support in a given $I\in\I$, and shortly $\Cci_c(\R):=\Cci_c(\R,\R)$ and $\S(\R):=\S(\R,\R)$.

We define the operator
\begin{equation}\label{eq:FF-P+}
 P_+: f\in \K \mapsto (\chi_{[0,\infty)}\cdot \hat{f})^\vee \in\K,
\end{equation}
where $\hat{f}, f^\vee\in\K$ denotes the Fourier transform (inverse Fourier transform, respectively) of $f\in\K$. Then $\phi_{P_+}$ is a quasi-free state, the vacuum state, and we call its GNS representation $\pi_{\phi_{P_+}}$ the \emph{vacuum representation} of $\CAR(\K,\Gamma)$. We fix the grading automorphism $\gamma$ and the translation automorphism group $\alpha$ by
\[
 \gamma(F(f))=-F(f), \quad \alpha_t(F(f))= F(f(\cdot-t)), \quad t\in\R, f\in\K.
\]
Let us consider (an adaptation of) the sKMS property $(S_2)$ on $(\CAR(\K,\Gamma),\gamma,\alpha)$:

\begin{proposition}\label{th:FF-RST}
The sesquilinear form
\begin{equation}\label{eq:FF-RST-2pt}
\theta(f,g)= \lim_{\eps\ra 0^+} \Big(\int_{-\infty}^{-\eps} + \int_\eps^\infty \Big) \frac{1}{1-\rme^{-p}}\overline{\hat{f}(p)} \hat{g}(p) \rmd p, 
\quad f,g\in\S(\R,\C^d),
\end{equation}
gives rise to a quasi-free functional $\phi$ on $\big(\CAR(L^2(\R,\C^d),\Gamma),\gamma,\alpha\big)$ with
\[
\dom(\phi):= \salg \{F(f):f\in \S(\R,\C^d) \},
\]
with two-point functions $\theta(f,g)$ and satisfying properties $(S_0),(S_2),(S_3),(S_6)$ with norm instead of $\sigma$-weak denseness. 
\end{proposition}
The proof is deferred to Appendix \ref{sec:proofs}.

The operators $F(f)\in \CAR(\K,\Gamma)$ define the \emph{$d$-fermion free field net $\F$ over $\R$}, namely
\[
 \F(I):= \{\pi_{P_+}(F(f)) : f\in \K_I^\Gamma \}'', \quad I\in\I,
\]
which is a graded-local conformal net, with grading and translation automorphisms the ones induced by $\gamma$ and $\alpha$, \cf \cite{Boc,CKL} for further information. Following our above convention, we write $\FF$ for the quasi-local C*-algebra of the net $\F$.

\begin{theorem}\label{th:FF-sKMS-ln}
There is a functional $(\phi_\F,\dom(\phi_\F))$ on $(\FF,\gamma,\alpha)$ satisfying $(S_0)$-$(S_3)$ and $(S_6)$. Furthermore, requiring 
quasi-freeness and 
\[
\dom(\phi_\F) = \salg \{ \pi_{P_+}(F(f)) : f\in\S(\R,\R^d) \},
\]
it is unique.
\end{theorem}

The proof is again deferred to Appendix \ref{sec:proofs}.

\subsection*{Supersymmetric free field net}

Consider the space $\S(\R,\R^d)$ and write $\L$ for the complex Hilbert space obtained by complexifying with the imaginary unit $\rmi_\L$ defined by $\rmi_\L f := (-\rmi \sgn \cdot \hat{f})^\vee $ and completing with respect to the complex scalar product
\[
\langle f,g\rangle_\L := 2 \sum_{i=1}^d \int_0^\infty p \overline{\hat{f}_i(p)} \hat{g}_i(p) \rmd p.
\]
Let $\L_I$ be the analogous complex completion of $\Cci_c(I,\R^d)$.
Consider the C*-algebra generated by $\{W(f): f\in\L\}$, satisfying the relations
\[
W(f)W(g)= \rme^{\frac{\rmi}{2}\Im \langle f,g\rangle_\L} W(f+g), \quad  W(f)^*= W(-f).
\]
We denote its Fock space representation by $(\pi_{0,\B},\H_\B)$. The \emph{$d$-bosonic free field} is generated by the Weyl operators $W(f)$ in the vacuum representation $\pi_{0,\B}$, namely
\[
 \B(I) = \{\pi_{0,\B}(W(f)): f\in\Cci_c(I,\R^d) \}'',\quad I\in\I,
\]
and it is a local conformal net (\ie with trivial grading) with translation automorphisms defined through $\alpha_t(\pi_{0,\B}(W(f))) = \pi_{0,\B}(W(f(\cdot-t)))$, for all $t\in\R, f\in\S(\R,\R^d)$, \cf \cite{BMT,FG,Xu1} for further information.  We remark that, for all $f\in\S(\R,\R^d)$, $t\in\R\mapsto \pi_{0,\B}(W(tf))\in B(\H_\B)$ is a strongly continuous one-parameter group of unitaries, and we write $J(f)$ for its (unbounded selfadjoint) infinitesimal generator. All $J(f)$ are linear in $f$ with a common invariant core as discussed below, on which they satisfy the canonical commutation relations
\[
[J(f),J(g)]= -\rmi \Im \langle f,g \rangle_\L = \rmi \int_\R f g', \quad f,g\in \S(\R,\R^d),
\]
\cf \cite[Sec.5.2]{BR2} and \cite[Sec.4]{CLTW2}.
Following our above convention, we write $\BB$ for the quasi-local C*-algebra corresponding to the net $\B$.

We let $\A:=\B\otimes\F$ denote the \emph{supersymmetric free field net} for the rest of this subsection, and we usually identify $\unit\otimes\F$ with $\F$, $\B\otimes\unit$ with $\B$, $\unit\otimes F(f)$ with $F(f)$ and  $J(f)\otimes\unit$ with $J(f)$. Moreover, we write $F(f)$ for $\pi_{\phi_{P_+}}(F(f))$ henceforth and $W(f)$ for $\pi_{0,\B}(W(f))$, all of them acting on the vacuum Hilbert space $\H=\H_\B\otimes\H_\F$. Moreover, they all have a common invariant core $\D_\infty$ (\cf \eg \cite{CLTW2}). Grading and translation automorphisms on $\AA$ are given by tensor product of the corresponding ones on $\BB$ and $\FF$, and we denote them again by $\gamma$ and $\alpha_t$.

It is well-known that \emph{$\A$ is a superconformal net over $\R$}. The interested reader may find the standard procedure described in \cite[Prop.6.2]{CHL}, which in turn is based on the super-Sugawara construction in \cite{KT} and on ideas in \cite[Sec.2]{BSM}; he only has to regard the involved formulae for fields over $\R$ instead of $S^1$ there, and moreover the involved Lie groups $\Uone^d$ instead of a simple simply-connected Lie group $G$ of dimension $d$ there. Some more details are provided also here below in the proof of Lemma \ref{lem:FF-delta}.

We would like to have a superderivation $\delta$ on $\A$ satisfying formally
\begin{equation}\label{eq:FF-delta-formal}
\delta(J(f))= \rmi F(f'), \quad \delta(F(f))= J(f), \quad f\in \Cci_c(\R,\R^d),
\end{equation}
or more precisely
\begin{equation}\label{eq:FF-delta-vNa}
 \begin{aligned}
\delta(W(f)) =& -F(f')W(f), \\ 
\delta(F(f)(J(f)+\rmi)^{-1})=& J(f)(J(f)+\rmi)^{-1} -\rmi F(f)F(f')(J(f)+\rmi)^{-2},
\end{aligned}
\end{equation}
on the generators of the *-subalgebra
\begin{equation}\label{eq:FF-A0}
\AA_0:=\salg \{W(f), F(f)(J(f)+\rmi)^{-1}: f\in \Cci_c(\R,\R^d)\}.
\end{equation}

\begin{lemma}\label{lem:FF-delta}
There is a family $(Q_I)_{I\in\I}$ of (unbounded) odd selfadjoint operators on $\H$ such that:
\begin{itemize}
\item[-] there is a corresponding family of linear maps $(\delta_{Q_I})_{I\in\I}$, namely $\dom(\delta_{Q_I})$ is the *-algebra given by all $x\in\A(I)$ for which there is  $y\in\A(I)$ with
\begin{equation}\label{eq:FF-dom-QI}
 \gamma(x) Q_I \subset Q_I x - y,
\end{equation}
and in this case $\delta_{Q_I}(x):=y$; 
\item[-] $\delta_{Q_I}$ satisfies \eqref{eq:FF-delta-vNa}, for all $f\in\Cci_c(I,\R^d)$, and $\delta_{Q_{I_2}}\restriction_{\dom(\delta_{Q_{I_1}})}= \delta_{Q_{I_1}}$, for all $I_1\subset I_2$, and $\AA_0\cap\A(I)\subset \dom(\delta_{Q_I})$;
\item[-] every $\delta_{Q_I}$ is ($\sigma$-weak)-($\sigma$-weakly) closed and satisfies
\[
\delta_{Q_I}(x_1 x_2) = \delta_{Q_I}(x_1)x_2 + \gamma(x_1) \delta_{Q_I}(x_2), \quad x_1,x_2\in\dom(\delta_{Q_I}).
\]
\end{itemize}
\end{lemma}

The proof, in Appendix \ref{sec:proofs}, is based on the fact that $\A$ is diffeomorphism-covariant and contains the super-Virasoro net as conformal subnet, which gives rise to the elements $Q_I$.

\begin{lemma}\label{lem:FF-xh}
Given the above family $(Q_I)_{I\in\I}$, an interval $I\in\I$ symmetric w.r.t.~$0$, $x\in \AA_0\cap\A(\frac12 I)$ and $h\in \Cci_c(\frac12 I)$, let
\[
x_h:=  \int_\R \alpha_t (x) h(t) \rmd t.
\]
Then $x_h\in\dom(\delta_{Q_I}^m)\cap \dom(\delta_0)$, for every $m\in\N$, and we have
\[
\delta_{Q_I}^{2n}(x_h)= (\rmi)^n x_{h^{(n)}}=\delta_0^n(x_h), \quad 
\delta_{Q_I}^{2n+1}(x_h) =(\rmi)^n (\delta_{Q_I}(x))_{h^{(n)}}.
\]
\end{lemma}

The proof is placed again in Appendix \ref{sec:proofs}.

For certain reasons to be explained later in the proof of Theorem \ref{th:FF-sKMS1}, $\dom(\delta_{Q_I})$ is too large for our purposes, while $\AA_0\cap\A(I)$ is too small since it intersects possibly trivially with $\Cci(\delta_{Q_I})$. A suitable domain is given by the following intermediate one, namely: setting
\[
A(I) := \salg \{x_h,(\delta(x))_h\in\A(I): x\in\AA_0\cap\A(I), h\in\Cci_c(I) \},
\]
we define
\begin{equation}\label{eq:FF-dom-deltaI}
\delta_I:= \textrm{($\sigma$-weak)-($\sigma$-weak)closure } (\delta_{Q_I}\restriction_{A(I)}),
\end{equation}
and
\begin{equation}\label{eq:FF-dom-delta}
\dom(\delta):= \bigcup_{I\in\I} \dom(\delta_I),
\end{equation}
which is clearly a *-algebra, since every $\dom(\delta_I)$ is so and $\dom(\delta_{I_1})\subset\dom(\delta_{I_2})$ if $I_1\subset I_2$. By construction, $\dom(\delta)=\dom(\delta)_c$.

\begin{proposition}\label{prop:JLO-delta2}
The family $(\delta_I)_{I\in\I}$ from \eqref{eq:FF-dom-deltaI} induces a unique superderivation $\delta$ on $\dom(\delta)$ satisfying conditions $(i)$-$(iv)$ and coinciding with $\delta_I$ on $\dom(\delta_I)$. Moreover,
$\AA_0\subset \dom(\delta)$, and $x_h\in \Cci(\delta_I)$ if $x\in\AA_0\cap\A(\frac12 I)$ and $h\in\Cci_c(\frac12 I)$, and $\delta_I^2\subset\delta_0$.
\end{proposition}

Since we find the proof quite instructive, we do not postpone it to the Appendix as usual.

\begin{proof}
Since $\delta_{I_1}\subset\delta_{I_2}$ if $I_1\subset I_2$, and every element $x\in\dom(\delta)$ lies in some local algebra $\A(I)$, $\delta$ acts on $x$ as the corresponding graded commutator $\delta_{Q_I}$ with $Q_I$, for which the general theory in \cite[Prop.2.1]{CHKL} applies. Thus, once we have shown the global $\alpha$-$\gamma$-invariance of the *-algebra $\dom(\delta)$, properties $(i)$ and $(ii)$ of superderivations follow immediately. 

To this end, $\alpha$-invariance of $\AA_0$ and its image $\delta(\AA_0)$ is clear from \eqref{eq:FF-A0} and the global translation invariance of $\Cci_c(\R,\R^d)$. Given $x\in\AA_0\cap \A(\frac12 I)$ and $h\in\Cci_c(\frac12 I)$ and $t\in\R$, let $J\in\I$ contain $I$ and $t+I$; then using the invariance of $\AA_0$, we find
from the definition of $x_h$ that $\alpha_t(x_h)=x_{h(\cdot-t)}$, which is now in $A(I+t)\subset A(J)$, since $h(\cdot-t)\in\Cci_c(J)$. Thus
\begin{equation}\label{eq:FF-alphainv}
\delta(\alpha_t(x_h))=\delta_J(\alpha_t(x_h))= \delta_J(x_{h(\cdot-t)})= (\delta_I(x))_{h(\cdot-t)} = \alpha_t(\delta_I(x)_h)=\alpha_t(\delta(x_h)),
\end{equation}
and analogously for the elements $\delta(x)_h$ in $A(I)$, so that we obtain the statement for all elements in the *-algebra $A(I)$. Finally, let $x\in\dom(\delta_I)$, for some $I\in\I$, and consider a sequence $x_n\in A(I)$ converging $\sigma$-weakly to $x$ and such that $\delta_I(x_n)$ converges $\sigma$-weakly to some element in $\A(I)$, which has to be $\delta_I(x)$ due to ($\sigma$-weak)-($\sigma$-weak) closedness of $\delta_I$, and in particular they converge $\sigma$-weakly. Then $\sigma$-weak continuity of $\alpha_t$ and $\gamma$ implies
\begin{align*}
\langle \eta,\gamma \alpha_t(x) Q_J \xi\rangle 
=& \lim_{n\ra \infty} \langle \eta,\gamma \alpha_t(x_n) Q_J \xi\rangle\\
=& \lim_{n\ra \infty}  \langle \eta,Q_J \alpha_t(x_n) \xi\rangle - \langle \eta,\delta_J(\alpha_t(x_n))\xi\rangle \\
=&  \lim_{n\ra \infty}  \langle Q_J\eta, \alpha_t(x_n) \xi\rangle -  \langle \eta,\alpha_t( \delta_{J-t}(x_n))\xi\rangle \\
=& \langle Q_J\eta, \alpha_t(x) \xi \rangle
- \langle \eta,\alpha_t(\delta(x))\xi\rangle\\
=& \langle \eta,Q_J \alpha_t(x) \xi \rangle
- \langle \eta,\alpha_t(\delta(x))\xi\rangle,
\end{align*}
for all $\eta,\xi\in \D_\infty$, which means that $\alpha_t(x)\in\dom(\delta)$ and $\delta(\alpha_t(x))= \alpha_t(\delta(x))$. This proves $\alpha$-invariance of $\dom(\delta)$; $\gamma$-invariance is proved analogously.

Let us prove property $(iv)$, \ie the $\sigma$-weak denseness of $\Cci(\delta_I)\subset\A(I)$. Owing to inner continuity, $\A(I)=\bigvee_{\bar{J}\subset I} \A(J)$, and by definition of $\A$ it is clear that $\AA_0\cap\A(J)$ is $\sigma$-weakly dense in $\A(J)$, for every $J$.  Thus it suffices to show that every $x\in\AA_0\cap\A(J)$ can be approximated $\sigma$-weakly by a sequence in $\Cci(\delta_I)$. To this end, let $h_n\in\Cci_c(-\frac{1}{n},\frac{1}{n})$ be a sequence of nonnegative functions such that the measures $h_n(t)\rmd t$ converge *-weakly on $C_b(\R)$ to the Dirac measure $\rmd_\delta t$ in $0$. Then $\sigma$-weak continuity of $t\mapsto \alpha_t(x)$ implies
\[
 x_{h_n} = \int_\R \alpha_t(x) h_n(t) \rmd t \ra x \quad \textrm{($\sigma$-weakly)},
n\ra\infty,
\]
and every $x_{h_n}$, with $n$ sufficiently large, lies in $\Cci(\delta_I)$ according to Lemma \ref{lem:FF-delta}. Since $\delta(x_{h_n})=\delta(x)_{h_n}\ra\delta(x)$, this proves, moreover, that $x\in\dom(\delta)$, so $\AA_0\subset\dom(\delta)$.

Property $(iii)$ holds by definition of $\delta_I$ in \eqref{eq:FF-dom-deltaI} and owing to the fact that the $\sigma$-weak closures of $A(I)$ and $\AA_0\cap\A(I)$ coincide according to the second paragraph and are in fact equal to $\A(I)$.

Finally, Lemma \ref{lem:FF-xh} showed that $\delta_I^2(x_h)= \delta_0(x_h)$ for $x_h\in A(I)$, so the final statement is a consequence of taking the closure.
\end{proof}

In order to make uniqueness statements for sKMS functionals in the present example, let us introduce now the following concept. Let us call a functional $\psi$ on $\BB\otimes\FF$ \emph{$\delta$-regular} if it splits into a quasi-free product functional $\psi_\B\otimes\psi_\F$, its domain contains the algebraic tensor product $\BB_0\odot \FF_0$, where
\begin{align*}
\FF_0 =& \salg \{F(f):f\in \S(\R,\R^d) \} \subset \FF,\\
\BB_0 =& \salg \{W(f), (J(f)+\rmi)^{-1}:f\in \Cci_c(\R,\R^d) \} \subset \BB,
\end{align*}
\cf Proposition \ref{th:FF-RST}, and $t\in\R^n\mapsto \psi(W(t_1f_1)\cdots W(t_nf_n))$ is analytic, for every $f_i\in\Cci_c(\R,\R^d)$ and $n\in\N$. More precisely, $\psi$ can be extended to the unbounded monomials of the form $J(f_1)\cdots J(f_n)F(f_{n+1})\cdots F(f_{n+m})$, for all $f_i\in\Cci_c(\R,\R^d)$ (for which\linebreak $F(f_i)(J(f_i)+\rmi)^{-1}\in\dom(\delta)$), and their image under $\delta$, namely:
\begin{align*}
\psi\big( J(f_1) &\cdots  J(f_n)F(f_{n+1})\cdots F(f_{n+m})\big)\\ 
:=& \rmi^{m-n}\frac{\rmd^{n+m}}{\rmd t_1\cdots \rmd t_{n+m}} \psi\Big( W(t_1f_1)\cdots W(t_nf_n)\cdot\\
&\cdot F(t_{n+1}f_{n+1})\rmi (J(t_{n+1}f_{n+1})+\rmi)^{-1}\cdots F(t_{n+m}f_{n+m})\rmi(J(t_{n+m}f_{n+m})+\rmi)^{-1}\Big)\restriction_{t=0},
\end{align*}
and analogously
\begin{align*}
\psi\circ\delta & \big(J(f_1)\cdots  J(f_n)F(f_{n+1})\cdots F(f_{n+m})\big)\\ 
:=& \rmi^{m-n}\frac{\rmd^{n+m}}{\rmd t_1\cdots \rmd t_{n+m}} \psi\circ\delta \Big( W(t_1f_1)\cdots W(t_nf_n)\cdot\\
&\cdot F(t_{n+1}f_{n+1})\rmi (J(t_{n+1}f_{n+1})+\rmi)^{-1}\cdots F(t_{n+m}f_{n+m})\rmi(J(t_{n+m}f_{n+m})+\rmi)^{-1}\Big)\restriction_{t=0},
\end{align*}
for every $f_i\in\Cci_c(\R,\R^d)$. We see that, under assumption $(S_4)$, $\delta$-regularity of $\psi$ implies
\[
 \psi\circ\delta \big( J(f_1)\cdots J(f_n)F(f_{n+1})\cdots F(f_{n+m}) \big) = 0.
\]

We come now to the main result of the present section:

\begin{theorem}\label{th:FF-sKMS1}
There is a unique nontrivial $\delta$-regular sKMS functional for the supersymmetric free field net $\A = \B \otimes \F$ with domain $\BB_0\odot\FF_0$. Furthermore, it satisfies $(S_6)$ and is given by the product of the geometric KMS state $\phi_\B$ on $\BB$ and the quasi-free functional $\phi_\F$ with two-point functions \eqref{eq:FF-RST-2pt} on $\FF$.
\end{theorem}

\begin{proof}
We first consider existence. Let $\psi=\phi_\B\otimes\phi_\F$ with $\dom(\psi):=\BB_0\odot\FF_0$ be the product of the unique geometric KMS state $\phi_\B$ on $\BB$ and the functional $\phi_\F$ constructed above on $\FF$. The local normality of $\phi_\B$ is a well-known fact because it is a KMS state \cite[Sec.3]{TW}, while local normality, sKMS property and quasi-freeness of $\phi_\F$ have been shown in Theorem \ref{th:FF-sKMS-ln}, which together with \cite[Prop.5.4]{BG} yield the corresponding properties of $\psi$. \eqref{eq:gen-S2b} of the sKMS functions and local-exponential boundedness of $\phi_\F$ have been established in \cite[Th.5.7]{BG}, and $\delta$-regularity of $\psi$ is a consequence of its definition as a product functional and the structure of the geometric KMS state $\phi_\B$ of $\BB$, \cf also \cite[Sec.4]{CLTW2}. This proved $(S_0)$-$(S_2)$, $(S_6)$, while $(S_3)$ is obvious.

Let us turn to derivation invariance $(S_4)$. In Appendix \ref{sec:proofs} we prove that
\[
\psi_I\delta_I (x) = 0, \quad x \in B(I):= \salg\{y, \delta(y): y\in\AA_0\cap\A(I)\},
\]
for every $I\in\I$. On elements of the form $x_{1,h_1}\cdots x_{n,h_n}$, with $x_i\in B(I)$ and $h_i\in\Cci_c(I)$ such that $x_{i,h_i}\in\A(I)$, we find then by straight-forward calculation (since $\alpha_{t_i}(x_i)\in B(I)$):
\begin{align*}
\psi_I\delta_I(x_{1,h_1}\cdots x_{n,h_n}) =& 
\int_{\R^n} \psi\delta (\alpha_{t_1}(x_1)\cdots \alpha_{t_n}(x_n)) h_1(t_1)\cdots h_n(t_n) \rmd^n t\\
=& \int_{\R^n} 0\cdot h_1(t_1)\cdots h_n(t_n) \rmd^n t=0.
\end{align*}
Thus $\psi_I\delta_I=0$ on the ($\sigma$-weak)-($\sigma$-weak) closure $\dom(\delta_I)$ owing to normality of $\psi_I$, so $\psi\delta=0$ on $\dom(\delta)_c$. Concerning $(S_5)$, recall from Proposition \ref{prop:JLO-delta2} that
\[
 \delta^2_I(y) = \delta_0(y) = -\rmi\frac{\rmd}{\rmd t}\alpha_t(y)\restriction_{t=0},\quad y\in\dom(\delta^2_I),
\]
so $(S_5)$ holds on every $\dom(\delta^m_I)$, $I\in\I$ and $m\in\N$, thus on $\Cci(\delta)_c$.

Uniqueness is more involved. The requirement of $\psi$ being a $\delta$-regular sKMS functional $\psi_\B\otimes\psi_\F$ forces the restriction to $\unit\otimes \FF_0$ to be a densely defined functional $\psi_\F$ with $\FF_0\subset \dom(\psi_\F)$ satisfying $(S_0)$-$(S_3)$, so it has to be $\phi_\F$ constructed in Theorem \ref{th:FF-sKMS-ln} as shown there. We have to determine $\psi_\B$. Owing to $\delta$-regularity we get
\begin{align*}
\psi (J(f_1)&\cdots J(f_{2n}))
= \rmi^{-2n}\frac{\rmd^{2n}}{\rmd t_1\cdots \rmd t_{2n}} \psi( W(t_1f_1)\cdots W(t_{2n}f_{2n}))\restriction_{t=0}\\
=& \rmi^{2-2n}\frac{\rmd^{2n}}{\rmd t_1\cdots \rmd t_{2n}} \Big(\psi\circ\delta
\Big( F(t_1f_1)(J(t_1f_1)+\rmi)^{-1} W(t_2f_2)\cdots W(t_{2n}f_{2n})\Big) \\
&- \sum_{k=2}^{2n} \psi \Big(F(t_1f_1)(J(t_1f_1)+\rmi)^{-1} W(t_2f_2)\cdots  F(t_kf_k')W(t_kf_k)\cdots  W(t_{2n}f_{2n})\Big)\Big)\restriction_{t=0}\\
=& 0 - \sum_{k=2}^{2n} \psi \big(F(f_1)F(f_k')\big)\cdot\\
&\cdot \rmi^{2-2n}\frac{\rmd^{2n}}{\rmd t_1\cdots \rmd t_{2n}}
\psi\Big((J(t_1f_1)+\rmi)^{-1} t_1 W(t_2f_2)\cdots t_k W(t_kf_k)\cdots  W(t_{2n}f_{2n})\Big)\restriction_{t=0}\\
=& \sum_{k=2}^{2n} \rmi \phi_\F \big(F(f_1)F(f_k')\big) \psi_\B\big(J(f_2)\cdots \check{J(f_k)}\cdots  J(f_{2n})\big)
\end{align*}
and by iteration 
\[
=\sum_{\sigma\in P_{2n}} \prod_{j=1}^n \rmi \phi_\F(F(f_{\sigma(j)}) F(f_{\sigma(n+j)}'))
= \sum_{\sigma\in P_{2n}} \prod_{j=1}^n \rmi \theta(f_{\sigma(j)},f_{\sigma(n+j)}'),
\]
where $P_{2n}=\{\sigma\in S_{2n}:\sigma(1)<...<\sigma(n), \sigma(j)<\sigma(j+n), j=1,...,n\}$.
By an analogous reasoning
\begin{align*}
\psi (J(f_1)\cdots J(f_{2n+1}))
= \sum_{k=1}^{2n+1} \sum_{\sigma\in P_{2n}^k} \prod_{j=1}^n \rmi \phi_\F(F(f_{\sigma(j)}) F(f_{\sigma(n+j)})) \cdot \psi_\B(J(f_{k})),
\end{align*}
where $P_{2n}^k=\{\sigma\in S_{2n+1}:\ \sigma(j)< \sigma(j+n+1), j=1,...,k-1;\ k=\sigma(k);\ \sigma(j)<\sigma(j+n), j=k,...,n+1\}$.
But
\[
 \psi_\B(J(f))= \psi(J(f))= \psi\big(\delta(F(f))\big)= 0,
\]
so $\psi_\B$ vanishes on all monomials with an odd number of factors $J(f)$. Thus, together with the above expression, we see that $\psi_\B$ is quasi-free, \ie for all $f_i\in\S(\R,\R^d)$,
\[
\psi_\B(J(f_1)\cdots J(f_{2n})) = \sum_{\sigma\in P_{2n}} \prod_{j=1}^n \psi_\B(J(f_{\sigma(j)}) J(f_{\sigma(n+j)})), 
\]
with 2-point function
\begin{equation}\label{eq:FF-1-2pt}
(f_1,f_2)\in \S(\R,\R^d)\times\S(\R,\R^d) \mapsto \psi_\B(J(f_1)J(f_2)) = \rmi \theta(f_1, f_2').
\end{equation}

The functional $\psi_\B$ is now completely determined on the unbounded smeared fields. By $\delta$-regularity (in particular analyticity), this defines a quasi-free functional $\psi_\B$ on the algebra generated by all $W(f)$, which is uniquely defined through its 2-point functions \eqref{eq:FF-1-2pt}. Since the latter coincide with those of the (quasi-free) geometric KMS state $\phi_\B$ on $\BB$ (\cf again \cite[Sec.4]{CLTW2}), we conclude that $\psi_\B$ is bounded and extends to the geometric KMS state $\phi_\B$ on $\BB$; therefore, $\psi=\phi_\B\otimes\phi_\F$ and $\dom(\psi)\supset\BB_0\odot\FF_0$.
\end{proof}

\subsection*{Examples through restriction}

Given the $d$-fermion free field $\F$ from above, let $\G\subset\F$ be an arbitrary conformal subnet. Recall from Proposition \ref{prop:gen-alpha-inv} that $\phi_\F\restriction_{\dom(\phi)_c}$ in Theorem \ref{th:FF-sKMS-ln} extends to a locally normal functional on the *-algebra $\bigcup_{I\in\I}\F(I)$, on which \eqref{eq:gen-S2a} still holds. Then one obtains a nontrivial unbounded functional on $\bigcup_{I\in\I}\G(I)$ simply by restriction $(\phi_\F,\dom(\phi_\F))$. It satisfies properties $(S_0)$-$(S_3)$ and $(S_6)$ except possibly for \eqref{eq:gen-S2b} by construction, in the same way as for the free fermion field. 

However, if we want compatibility with respect to a superderivation $\delta$ on $\G$, note that $\delta$ has in general nothing to do with $\phi_\F$ but depends instead only on the subnet; hence, we cannot expect the compatibility conditions $(S_4)$-$(S_5)$ to hold in this case. 

The most important and fundamental (ungraded) local conformal nets are the Virasoro nets $\A_{\Vir,c}$ with certain central charges $c>0$ because every conformal net contains a copy of one of them, \cf \eg \cite{Carpi}; their KMS states are discussed in \cite[Sec.5]{CLTW2}. For $c<1$,  $\A_{\Vir,c}$ is completely rational, and for the particular value $c=1/2$, it coincides with the fixed point subnet $\F^\gamma$ of the $d$-fermion free field with $d=1$. In this case we do not have to pass to $\bigcup_{I\in\I}\F(I)$ but instead from Theorem \ref{th:FF-sKMS-ln} we immediately obtain

\begin{corollary}\label{cor:FF-even-subnet}
Besides the unique (geometric) KMS state, the completely rational local conformal net $\A_{\Vir,1/2}$ has a nontrivial nonpositive unbounded KMS functional.
\end{corollary}

\begin{proof}
Let $d=1$ and restrict the functional $\phi_\F$ in Theorem \ref{th:FF-sKMS-ln} to $\FF^\gamma\cap\dom(\phi_\F)=\AA_{\Vir,1/2}\cap\dom(\phi_\F)$. It is unbounded and nonpositive and satisfies $(S_0)$-$(S_3)$, as can be easily checked based on the preceding results and local normality of $\phi_\F$. In restriction to even elements, the sKMS condition becomes the KMS condition. Since $\F^\gamma=\A_{\Vir,1/2}$ is completely rational, it has a unique KMS state $\phi_{\geo}$ (\cf \cite{CLTW1}), but furthermore, as just seen, at least one nontrivial (locally bounded and locally normal) unbounded KMS functional $\phi_\F \restriction_{\FF^\gamma} \not=\phi_{\geo}$.
\end{proof}

\subsection*{Super-Virasoro nets. (Sketch)}

Returning to superconformal nets, consider the supersymmetric free field net $\A=\B\otimes\F$ discussed above, with $d=1$ here, and as always on the vacuum space $\H$. As mentioned before, $\A$ is superconformal, \ie it contains the super-Virasoro net $\A_{\SVir,c}$ with central charge $c=3/2$ as diffeomorphism-covariant subnet, with super-Virasoro generators $L(f), G(f)$, for $f\in\Cci_c(\R)$, as introduced in \eqref{eq:FF-SVir} and \eqref{eq:FF-SVirFF} in Appendix \ref{sec:proofs}. The construction in \cite[(4.6)]{BSM} and \cite[p.793]{CLTW2} inspires moreover to the following modification: for every $s\ge 0$ and $f\in\Cci_c(\R)$, let
\begin{equation}\label{eq:LsGs}
L_s(f) := L(f) + s J(f'), \quad G_s(f) := G(f) + 2 s F(f').
\end{equation}
Since this example is not of our main interest, we skip the lengthy analytical details, and just keep on record that the new fields $L_s,G_s$ satisfy again the super-Virasoro commutation relations \eqref{eq:FF-SVir}, but now with central charge $c_s=3/2 + 12 s^2$. They give in fact rise to the super-Virasoro net $\A_{\SVir,c_s}$. It has the same grading $\gamma$ and translation action $\alpha$ as $\A$ and $\A_{\SVir,c_s}(I)\subset\A(I)$, for every $I\in\I$, but for $s\not=0$ it is not rotation-covariant, hence not a conformal subnet of $\A$ in the strict sense.

All $\A_{\SVir,c_s}(I)$ inherit from $\A(I)$ the superderivation $\delta_{Q_I}$  constructed in Lemma \ref{lem:FF-delta} with $Q_I=G(\varphi_I)$. It follows from the graded commutation relations \eqref{eq:FF-SVirFF}, \eqref{eq:FF-SVir} and the definition of $L_s$ and $G_s$ that, for every $f\in\Cci_c(I)$,
\[
G_s(\varphi_I)X_s(f) \xi -\gamma(X_s(f)) G_s(\varphi_I)\xi = 
G(\varphi_I)X_s(f)\xi - \gamma(X_s(f))G(\varphi_I)\xi, \quad \xi\in \D_\infty,
\]
where $X$ stands for $L,G$. Thus replacing $Q_I$ with $G_s(\varphi_I)$ leads to the same superderivation in restriction to $\dom(\delta_{Q_I})\cap\A_{\SVir,c_s}(I)$; this in turn enables us to apply the procedure
in \cite[Sec.5]{CHKL} in order to show that $\dom(\delta_{Q_I})\cap \A_{\SVir,c_s}(I)\subset\A_{\SVir,c_s}(I)$ is $\sigma$-weakly dense.
However, recalling \eqref{eq:FF-dom-deltaI}, we cannot even say whether $\dom(\delta_I) \cap \A_{\SVir,c_s}(I) \not= \C\unit$. Since $\psi_I$ is constructed w.r.t.~$\delta_I$ and not $\delta_{Q_I}$, we have to choose for the superderivation $\delta_{\SVir,c_s}$ the restriction of $\delta$ to the super-Virasoro net, \ie
\[
 \dom(\delta_{\SVir,c_s})= \bigcup_{I\in\I} \dom(\delta_{\SVir,c_s})_I, \quad \dom(\delta_{\SVir,c_s})_I= \dom(\delta_I)\cap\A_{\SVir,c_s}(I),
\]
whence $(S_4)$ continues to hold but the denseness properties in $(iii)$ and $(iv)$ might unfortunately fail for the present model. Alternatively, one might construct a superderivation out of $\delta_{Q_I}$ as in \eqref{eq:FF-dom-deltaI} but with respect to a different algebra $A(I)$ which lies in $\A_{\SVir,c_s}(I)$ so that $(iii)$ and $(iv)$ are verified; then, however, compatibility with $\psi_I$, \ie $(S_4)$ might become a problem. We expect this to be a very technical issue, which we have not been able to treat so far. Nevertheless, we can say

\begin{theorem}\label{th:FF-subnet}
Let $\A_{\SVir,c}$ be the super-Virasoro net, with $c\ge 3/2$. Then the corresponding restricted quantum dynamical system $(\AA_{\SVir,c},\gamma,\alpha,\delta_{\SVir,c})$ constructed above has at least one nontrivial sKMS functional, namely the restriction of the unique $\delta$-regular sKMS functional $\psi$ of the free field, satisfying $(S_0)$-$(S_6)$ but possibly not \eqref{eq:gen-S2b} and $\delta_{\SVir,c}$ may not satisfy $(iii)$ and $(iv)$.
\end{theorem}

\begin{proof}
Given $c\ge 3/2$, there is a unique $s\ge 0$ such that $c=c_s$ and we can carry out the above construction. Let $\psi$ be the sKMS functional of the free field net $\A$ from Theorem \ref{th:FF-sKMS1}, and notice that we may replace $\dom(\psi)$ by $\bigcup_{I\in\I}\A(I)$ according to Proposition \ref{prop:gen-alpha-inv}, possibly dispensing with \eqref{eq:gen-S2b}. Owing to local normality and local boundedness of $\psi$, it restricts to a nontrivial locally normal and locally bounded functional on $\AA_{\SVir,c}$. Since the translation action $\alpha$ and the superderivation $\delta_{\SVir,c}$ are just restriction to that subnet, we have $\im( \delta_{\SVir,c,I})\subset \dom(\psi \restriction_{\A_{\SVir,c}(I)})$; the remaining properties of sKMS functionals are also a trivial consequence of the restriction procedure.
\end{proof}

One might also try to use the above ideas in order to construct examples in higher dimensions, \eg as restrictions of product nets with product functionals. However, this will be quite involved and goes definitely beyond the scope of the present article.

\section{Super-KMS functionals for graded-local conformal nets over $S^1$}\label{sec:S1}

Let $\B_S$ be a completely rational \cite{KLM} local conformal net over $S^1$ and $C^*_{\locn}(\B_S)$ its locally normal universal C*-algebra \cite[Sec.3]{CCHW}. It was shown in \cite[Th.3.3]{CCHW} that $C^*_{\locn}(\B_S)\simeq B(\H)^{\oplus n}$, where $n$ is the number of equivalence classes of irreducible locally normal representations of $\B_S$. We write $\iota_I:\B_S(I)\ra C^*_{\locn}(\B_S)$ for the natural inclusions, but often drop the symbol $\iota_I$ when no confusion is likely. A natural symmetry group action here comes from the rotation group $\mathbb{T}$, whose lift to $\R\simeq \mathbb{T}^{(\infty)}$ has as generator the conformal Hamiltonian $L_0$.
Its action by *-automorphisms, however, is periodic and not asymptotically graded-abelian, so we cannot expect any clustering property here. Suppose that $\rme^{-\beta L_0^\pi}$ is trace-class, for every $\beta>0$ and every irreducible locally normal representation $\pi$ of $\B_S$. It is known that in this setting and for every such $\pi$, every KMS state on $\pi(C^*(\B_S))= B(\H_\pi)$ has to be the Gibbs state $\mu_\pi\tr_{\H_\pi}(\cdot \rme^{-L_0^\pi})$ (with $\mu_\pi$ a normalization factor),
\cf \cite[Sec.V.1]{Haag} together with \cite[Th.2.11]{CCHW}.

In the case of a graded-local net $\A_S$ over $S^1$, we shall find a similar result, with a similar but more involved argument, which we would like to present here. The representation structure of $\A_S$ in relation to $\A_S^\gamma$ has been determined in \cite[Sec.4.3]{CKL}, from where we are going to recall a few facts now. A \emph{general representation} of $\A_S$ is a locally normal representation of the double covering net $\A_S^{(2)}$ over $\I^{(2)}$ which restricts to a proper locally normal representation of the even subnet $\A_S^\gamma$, and we denote the irreducible general representations by $\Delta$; $\I^{(2)}$ stands for the set of open intervals of $(S^1)^{(2)}$ which project onto nondense nonempty open intervals of $S^1$. There are two kinds of irreducible general representations of $\A_S$: Neveu-Schwarz and Ramond general representations, the former ones giving rise to proper representations of $\A_S$, the latter ones only to solitons of $\A_S$. 

\emph{Suppose henceforth $\A_S$ is a graded-local conformal net over $S^1$ such that $\A_S^\gamma$ is completely rational and $\rme^{-\beta L_0^\pi}$ is trace-class, for every $\beta>0$ and every $\pi\in\Delta$.} This is the case for many models, \eg the supercurrent algebra and discrete series super-Virasoro models studied in \cite[Sec.6]{CHL} and \cite{CHKLX}.
Restricting a $\pi\in\Delta$ to $\A_S^\gamma$, we obtain either a direct sum of two inequivalent irreducible representations $\rho$ and $\rho\hat{\gamma}$ of $\A_S^\gamma$, with $\hat{\gamma}$ the dual of the grading, or one irreducible representation $\rho\simeq\rho\hat{\gamma}$, depending on whether $\pi$ was graded or ungraded, respectively (\cf \cite[Sec.2.5]{CKL} or \cite[Sec.2]{CHL}). 
V.v., extending an irreducible representation of $\A_S^\gamma$ to $\A_S^{(2)}$ by $\alpha$-induction, it is either irreducible or the direct sum of two ungraded general representations which restrict to the same representation of $\A_S^\gamma$. One then realizes that if there are $m_1$ graded and $m_2$ ungraded, so a total of $m=m_1+m_2$, equivalence classes in $\Delta$, then $\A_S^\gamma$ has $n=2m_1+\frac12 m_2$ equivalence classes of irreducible representations. We denote the irreducible graded general representations by $\Delta_\gamma$. Neveu-Schwarz general representations are always graded while Ramond may be graded or ungraded.

Exactly as in the local case \cite[Sec.5]{FRS}, one can constructively define the universal C*-algebra $C^*(\A_S^{(2)})$ of the net $\A_S^{(2)}$ with $\I_S$ replaced by $\I_S^{(2)}$, which is generated by all local algebras $\A_S^{(2)}(I)$ and has the universal property that every representation of $\A_S^{(2)}$ extends to a unique representation of $C^*(\A_S^{(2)})$. Let us consider the \emph{universal locally normal general representation}, \ie the direct sum representation $\pi_{\locn}$ of all GNS representations of $C^*(\A_S^{(2)})$ which correspond to general representations of the net $\A_S$, and let us call $C^*_{\locn}(\A_S):=\pi_{\locn}(C^*(\A_S^{(2)}))$ the \emph{universal locally normal C*-algebra of the graded-local net $\A_S$}.
We continue to use the same symbol for representations of $\A_S^{(2)}$ and of $C^*_{\locn}(\A_S)$. Notice that $\pi_{\locn}$ is quasi-equivalent to the direct sum of irreducible GNS representations with only one representative for each equivalence class, \ie of $m$ mutually inequivalent elements in $\Delta$. We then obtain the following modification of \cite[Th.3.3(3)]{CCHW}:

\begin{proposition}\label{prop:gen-ratCA}
Let $\A_S$ be a graded-local conformal net as above. Then $C^*_{\locn}(\A_S)$ is weakly closed and hence $C^*_{\locn}(\A_S) \simeq B(\H)^{\oplus m}$, with $m$ the number of equivalence classes in $\Delta$.
\end{proposition}

\begin{proof}
Let $\pi\in\Delta$, so $\H_\pi\simeq\H$. If it is ungraded, then $\rho:=\pi\restriction_{\A_S^\gamma}$ on $\H_\rho=\H_\pi$ is irreducible again and
\[
B(\H_\pi)= B(\H_\rho) = \pi\restriction_{\A_S^\gamma}(C^*(\A_S^\gamma)) \subset \pi(C^*(\A_S^{(2)}))\subset B(\H_\pi),
\]
the second equality being proved in \cite[Th.3.3(3)]{CCHW}, so $\pi(C^*(\A_S^{(2)}))= B(\H_\pi)$. If $\pi$ instead is graded, then $\pi\restriction_{\A_S^\gamma}$ decomposes into the direct sum $\rho\oplus\rho\hat{\gamma}$ of two irreducible representations, so
\[
B(\H_{\rho})\oplus B(\H_{\rho\hat{\gamma}}) = \pi\restriction_{\A_S^\gamma}(C^*(\A_S^\gamma))\subset \pi(C^*(\A_S^{(2)})) \subset B(\H_\pi).
\]
Furthermore, given $I\in\I_S^{(2)}$, $\A_S^{(2)}(I)$ is generated by $\A_S^\gamma(I)$ and any \emph{odd} selfadjoint $v\in\A_S^{(2)}(I)$ (\cf\cite[Sec.2.6]{CKL}), so
\[\pi(v)= \begin{pmatrix}
 0 & v_1 \\ v_1^* & 0 
\end{pmatrix} \in B(\H_\pi)
\]
with some $v_1$. Since $\pi(v)$ and $B(\H_{\rho})\oplus B(\H_{\rho\hat{\gamma}})$ together generate $B(\H_\pi)= B(\H_{\rho}\oplus\H_{\rho\hat{\gamma}})$, we find $\pi(C^*(\A_S^{(2)}))= B(\H_\pi)$.

The universal locally normal representation of $C^*(\A_S^{(2)})$ is quasi-equivalent to the direct sum of $m_1$ graded and $m_2$ ungraded mutually inequivalent irreducible representations, so that $C^*_{\locn}(\A_S)''$ is isomorphic to $B(\H)^{\oplus(m_1+m_2)}$. If we can show that the central support projections $s(\pi)\in C^*_{\locn}(\A_S)''$ of those $m_1+m_2$ representations lie actually in $C^*_{\locn}(\A_S)$, then we get (\cf the proof of \cite[Th.2.13]{CCHW} for details) $C^*_{\locn}(\A_S) \simeq B(\H)^{\oplus(m_1+m_2)}$.

To this end, let $\pi\in\Delta$. If it is graded, then we have $s(\pi)=s(\rho) + s(\rho\hat{\gamma})\in C^*_{\locn}(\A_S^\gamma)$ owing to \cite[Prop.2.12]{CCHW}, so our claim follows. On the other hand, for ungraded $\pi$, $\pi\gamma$ is a different irreducible representation of $C^*(\A_S^{(2)})$ restricting to the same irreducible representation $\rho$ of $C^*(\A_S^\gamma)$, and $\pi':=\pi\oplus \pi\gamma$ becomes a (reducible) representation graded by 
$\hat{\Gamma}_{\pi'}= \begin{pmatrix}
0 & \unit \\ \unit & 0
\end{pmatrix},$ 
and $s(\rho)\in C^*_{\locn}(\A^\gamma)$, so $s(\pi')\in C^*_{\locn}(\A_S)$. We have to show that
$\begin{pmatrix}
\unit &0 \\ 0 & 0
\end{pmatrix} \in \pi'(C^*(\A_S^{(2)}))$.
To this end, notice that, for the odd selfadjoint $v\in C^*(\A_S^{(2)})$ from above and any even $w\in C^*(\A_S^{(2)})$, we have $v_2,w_2\in B(\H_\pi)$ such that
\begin{equation*}
\pi'(v)=\begin{pmatrix}
v_2 &0 \\ 0 & - v_2
\end{pmatrix} \in \pi'(C^*(\A_S^{(2)})),
\quad
\pi'(w)= \begin{pmatrix}
w_2 &0 \\ 0 &  w_2
\end{pmatrix} \in \pi'\restriction_{\A_S^\gamma} (C^*(\A_S^\gamma)).
\end{equation*}
Since the image of $\pi'\restriction_{\A_S^\gamma}$ in $B(\H_{\pi'})$ is $\{x\oplus x: x\in B(\H_\pi)\}$ according to \cite[Th.2.11]{CCHW} and since $v_2$ is selfadjoint with infinite-dimensional range, we may choose $w$ such that $w_2 v_2 w_2^*=\unit_{B(\H_\pi)}$. Thus $\frac12\pi'(\unit + wvw^*)$ is the desired element and
\[
s(\pi)= s(\pi')\frac12 \pi_{\locn}(\unit + wvw^*) \in C^*_{\locn}(\A_S).
\]
\end{proof}

We would now like to study \emph{bounded sKMS functionals on $C^*_{\locn}(\A_S)$ with respect to the rotation group action $\alpha$}, \ie locally normal bounded selfadjoint functionals $\phi$ on $\AA:=C^*_{\locn}(\A_S)$ satisfying properties $(S_0)$-$(S_3)$ except for possibly \eqref{eq:gen-S2b}, with $\dom(\phi)=\AA$, with $\I_S$ instead of $\I$ and with $\alpha$ the canonical rotation action on $\A_S$, and dispensing with $(S_4)$ and $(S_5)$ for the moment.

Notice that the obstruction in Proposition \ref{prop:gen-obstruction} does not hold here since the rotation group action is not asymptotically graded-abelian, so bounded sKMS functionals may exist here, and in fact they do! For the proof we shall need

\begin{lemma}\label{lem:gen-KMSnormal}
Let $\B_S$ be a completely rational local conformal net and $\psi$ a locally normal KMS state on $C^*_{\locn}(\B_S)$ w.r.t. rotations $\alpha$. Then $\psi$ is normal.
\end{lemma}

\begin{proof}
In order to show normality of $\psi$, it suffices to prove separability of the GNS representation $(\H_\psi,\pi_\psi,\xi_\psi)$ \cite[Th.5.5.1]{Tak}. 

To this end, given $I\in\I_S$, consider $\H_{\psi_I}:= \overline{\pi_{\psi_I}(\B_S(I))\xi_\psi}= \overline{\pi_{\psi}(\B_S(I))\xi_\psi}\subset \H_\psi$, which is separable because $\psi$ is \emph{locally} normal by assumption. If we suppose $\H_{\psi_I}=\H_\psi$ then we are done, so let us show the separability of $\H_\psi$ assuming it is strictly larger than $\H_{\psi_I}$. Then there is $\eta\in \H_\psi \ominus \H_{\psi_I}$, implying that, for every $I_0\in\I_S$ with $\bar{I}_0\subset I$ and for every $x\in\B_S(I_0)$ and every $s\in\R$ in a connected $0$-neighborhood such that $s+I_0\subset I$, we have 
\[
\langle\eta, \pi_{\psi}(\alpha_s(x)) \xi_\psi \rangle = 0.
\]
The KMS condition for $\psi$ implies that this is the boundary value of a function analytic on the strip and continuous on the closure of the strip, which must therefore be identically $0$. This means that, for every $\zeta\in I'$ and with $\I^\zeta:=\{J\in\I_S:\zeta\not\in\bar{J}\}$ and $\BB_S^\zeta$ the quasi-local C*-algebra generated by $\{\B_S(J): J\in\I^\zeta\}$, the subspace     $\H_\psi^\zeta:=\overline{\pi_\psi(\BB_S^\zeta) \xi_\psi}\subset \H_\psi$ is rotation-invariant and therefore independent of $I$ and of $\zeta$ (which changes under the action of the rotation group), and thus invariant under the action of all local algebras and hence of $\pi_\psi(C^*_{\locn}(\B_S))$. Consequently, $\H_\psi^\zeta=\H_\psi$. Since $\H_\psi^\zeta$ is separable according to \cite[Lem.2.1]{CLTW1}, we get the separability of $\H_\psi$.
\end{proof}

We write $\pi_i$, $i=1,\ldots, m_1$, for an arbitrary fixed family of mutually inequivalent representatives of the $m_1$ equivalence classes in $\Delta_\gamma$ and $\hat{\Gamma}_{\pi_i} \in \pi_i(C^*(\A_S^{(2)}))$ for the grading unitary of $\pi_i$.

\begin{proposition}\label{prop:gen-Gibbs}
Let $\A_S$ be a graded-local conformal net as above. Then the bounded sKMS functionals on $C^*_{\locn}(\A_S)$ with respect to the rotation group are precisely those of the form
\begin{equation}\label{eq:S1-phi}
 \phi_\mu = \sum_{i=1}^{m_1} \mu_i \tr_{\H_{\pi_i}}(\hat{\Gamma}_{\pi_i} \pi_i(\cdot) \rme^{-L_0^{\pi_i}}),
\end{equation}
with any $\mu$ in the hyperplane 
\[
H=\Big\{\mu\in\R^{m_1}: \sum_{i=1}^{m_1} \mu_i\tr_{\H_{\pi_i}}(\hat{\Gamma}_{\pi_i} \rme^{-L_0^{\pi_i}})=1\Big\}\subset\R^{m_1}.
\]
\end{proposition} 

\begin{proof}
Let $\phi$ be an sKMS functional on $C^*_{\locn}(\A_S)$ (so by definition bounded and locally normal). Let us write $\phi_i$ for the restriction of $\phi$ to the $i$-th graded direct summand in $C^*_{\locn}(\A_S)$. Given the structure of $C^*_{\locn}(\A_S^\gamma)$ and $C^*_{\locn}(\A_S)$ determined above (in particular in the proof of Proposition \ref{prop:gen-ratCA}), it follows that $\phi$ restricts to a locally normal KMS functional on $C^*_{\locn}(\A_S^\gamma)$, whose GNS representation according to Lemma \ref{lem:gen-KMSnormal} is separable, thus the GNS representation of $\phi_i\restriction_{\pi_i(C^*(\A_S^\gamma))}$ is separable. As $\pi_i(C^*(\A_S^{(2)})) = \pi_i(C^*(\A_S^\gamma))\oplus \begin{pmatrix}
0 & \unit \\ \unit & 0
\end{pmatrix}\pi_i(C^*(\A_S^\gamma))$, it follows that the GNS representation of $\phi$ is separable, so $\phi_i$ on $\pi_i(C^*(\A_S^{(2)}))$ is normal. With polar decomposition as in the proof of \cite[Prop.4]{BL00} we find that $|\phi_i|$ is a positive normal KMS functional on $\pi_i(C^*(\A_S^{(2)}))=B(\H_{\pi_i})$ and hence a multiple of a Gibbs state, \ie  
of the form $\lambda_i\tr_{\H_{\pi_i}} (\cdot \rme^{-L_0^{\pi_i}})$ with some $\lambda_i>0$, \cf \cite[Sec.V.1]{Haag}.
Therefore $\phi_i= \lambda_i\tr_{\H_{\pi_i}} (v_i \cdot\rme^{-L_0^{\pi_i}})$ with $v_i$ a partial isometry in $B(\H_{\pi_i})$.

It remains to determine $v_i$. Letting $\hat{\Gamma}_{\pi_i}$ denote the grading unitary in the representation $\pi_i$ (which commutes with $\rme^{-L_0^{\pi_i}}$), the sKMS condition $(S_2)$ on $\phi_i$ together with cyclicity of the trace implies
\[
 \tr_{\H_{\pi_i}}(v_i x y  \rme^{-L_0^{\pi_i}}) = \tr_{\H_{\pi_i}} (v_i \hat{\Gamma}_{\pi_i} y\hat{\Gamma}_{\pi_i} \rme^{-L_0^{\pi_i}}x)
= \tr_{\H_{\pi_i}} (\hat{\Gamma}_{\pi_i} x v_i \hat{\Gamma}_{\pi_i} y \rme^{-L_0^{\pi_i}}), \quad x,y\in B(\H_{\pi_i}),
\]
so $\hat{\Gamma}_{\pi_i}v_i x\hat{\Gamma}_{\pi_i} = x v_i$. Choosing for $x$ \eg the standard system of matrix units in $B(\H_{\pi_i})$, we find $v_i=\pm\hat{\Gamma}_{\pi_i}$, so $\phi_i$ is a super-Gibbs functional.

Suppose instead $\pi\in\Delta$ is ungraded, then $\pi\gamma$ is a different ungraded irreducible representation which restricts to the same representation of $\A^\gamma$, and $\pi':=\pi\oplus\pi\gamma$ becomes a reducible graded representation with grading unitary 
$\hat{\Gamma}_{\pi'}=\begin{pmatrix} 0 & \unit \\ \unit & 0 \end{pmatrix}$.
Following then an analogous argument as above based on the sKMS condition and with the restriction of $\phi$ to $\pi'(C^*(\A_S^{(2)}))$, we get normality and find that the corresponding isometry $v_{\pi'}\in B(\H_{\pi'})$ must be $\hat{\Gamma}_{\pi'}$; but since this is traceless, it follows that $\phi$ vanishes on $\pi'(C^*(\A_S^{(2)}))$, hence on $\pi(C^*(\A_S^{(2)}))$ and likewise on all ungraded direct summands of $C^*_{\locn}(\A_S)$.

These two parts together prove that $\phi$ has the form \eqref{eq:S1-phi}.
The hyperplane structure and the condition on $\mu$ follow from the requirement that $\phi(\unit)=1$ which is possible since all $\tr_{\H_{\pi_i}}(\hat{\Gamma}_{\pi_i}  \rme^{-L_0^{\pi_i}}) $ are finite, and nonzero at least when $\pi_i$  is the vacuum representation. 

The converse is clear since every $\phi_\mu$ is a linear combination of super-Gibbs functionals. The fact that super-Gibbs functionals satisfy the conditions of bounded sKMS functionals is seen as in the case of KMS states, \cf \eg \cite[Sec.V.1]{Haag}.
\end{proof} 

We finally remark that, if $\A_S$ is actually superconformal, then the restriction of an sKMS functional to the summands in $C^*_{\locn}(\A_S)$ corresponding to Ramond general representations $\pi$ has a canonical superderivation, namely the one implemented by the supercharge $G_0^\pi$ in \cite[Sec.4]{CHKL}, verifying in fact properties $(S_4)$ and $(S_5)$. For the Neveu-Schwarz part this is never the case, \cf \cite[Sec.5]{CHKL}. The graded Ramond representations are therefore the ones of interest in the context of superderivations and sKMS functionals. This found an application \eg in \cite[Sec.4]{CHL} and other works.

\appendix

\section{A generalization of Araki's criterion on quasi-equivalence}\label{sec:Araki}

Let $\K$ be a complex Hilbert space with selfadjoint involution $\Gamma$ and use the notation for quasi-free states as at the beginning of Section \ref{sec:FF}, and write $L^1(\K)\subset B(\K)$ for the ideal of trace-class operators on $\K$. Recall that two representations $\pi_1,\pi_2$ of a C*-algebra $A$ are called quasi-equivalent if there is a *-isomorphism $\Psi:\pi_1(A)''\ra\pi_2(A)''$ such that $\Psi\circ\pi_1=\pi_2$. 

\begin{theorem}[Araki \cite{Ara}]
Given  two selfadjoint operators $R,S\in B(\K)$ such that $0\le R\le \unit, 0\le S \le \unit$ and $R+\Gamma R\Gamma = S+\Gamma S\Gamma=\unit$, assume that
\[
(S^{1/2}-R^{1/2})^2 \in L^1(\K). 
\]
Then the GNS representations $\pi_{\phi_R}$ and $\pi_{\phi_S}$ of the two quasi-free states $\phi_R$ and $\phi_S$ are quasi-equivalent, and we also say $\phi_R$ and $\phi_S$ are quasi-equivalent, denoted $\pi_1\simeq\pi_2$.\label{th:FF-Araki}
\end{theorem}

Recall the following useful reformulation of \cite[Prop.10.3.13]{KR2}:

\begin{lemma}\label{lem:FF-qe-ql}
Let $\AA$ be a C*-algebra. If $\phi,\psi$ are two states on $\AA$ such that $\pi_\phi$ and $\pi_\psi$ are quasi-equivalent, then $\psi$ is normal in the GNS representation $\pi_\phi$, and vice versa. More precisely, there exists a normal state $\tilde{\psi}$ of $\pi_\phi(\AA)''$ such that $\psi=\tilde{\psi}\circ \pi_\phi$.
\end{lemma}

The last part of this lemma makes also sense when $\psi$ is no longer a state. We may then ask whether it is also possible to adapt Theorem \ref{th:FF-Araki} to the case when $\psi=\phi_S$ is no longer a state. In a weak (but for our purposes sufficient) version this is in fact possible:

\begin{theorem}\label{prop:FF-comm2}
Given  two selfadjoint operators $R,S\in B(\K)$ such that $0\le R\le \unit$ and $R+\Gamma R\Gamma = S+\Gamma S\Gamma=\unit$, assume that
\[
(S-R) \in L^1(\K). 
\]
Then the quasi-free functional $\phi_{S}$ on $\CAR(\K,\Gamma)$ is bounded and normal in the GNS representation of the quasi-free state $\phi_R$.
\end{theorem}

We shall use and prove Theorem \ref{prop:FF-comm2} in the following equivalent formulation:

\begin{theorem}\label{prop:FF-comm}
Given  two selfadjoint operators $T\in L^1(\K)$ and $R\in B(\K)$ such that
\[
 T + \Gamma T \Gamma =0, \quad R+ \Gamma R \Gamma = \unit, \quad 0\le R\le \unit.
\]
Then the quasi-free functional $\phi_{R+T}$ on $\CAR(\K,\Gamma)$ is bounded and normal in the GNS representation of the quasi-free state $\phi_R$.
\end{theorem}

\begin{proof}
(1) \emph{Decomposition of $\K$}. Let $S := R+T\in B(\K)$;  it is selfadjoint and thus gives rise to a decomposition of $\K$ into spectral subspaces: letting 
\[
P_0 := \chi_{(-\infty,0)}(S), \quad P_1 := \chi_{[0,1]}(S),\quad P_2 := \chi_{(1,\infty)}(S),
\]
the projections $P_i$ are mutually orthogonal and $P_0+P_1+P_2=\unit$. We set $\K_i:= P_i\K$, so $\K= \K_0\oplus \K_1\oplus \K_2$ and split $S$ accordingly into the sum of
\[
S_0:=SP_0 \le0,\quad  0 \le S_1:=S P_1 \le P_1,\quad P_2 < S_2:=S P_2.
\]
The operator $\tilde{S}:=\Gamma S \Gamma = \unit-S$ has the same spectrum as $S$ and gives therefore rise again to spectral projections $\tilde{P}_i=\Gamma P_i \Gamma$. We find
\[
\Gamma P_0 \Gamma = \chi_{(-\infty,0)}(\Gamma S \Gamma) = \chi_{(-\infty,0)}(\unit - S) =  \chi_{(1,\infty)}(S) = P_2
\]
and analogously $\Gamma P_1 \Gamma = P_1$, so $\Gamma\K_0= \K_2$ and $\Gamma \K_1=\K_1$. Moreover,
\[
\tilde{S}_2 + S_0 = P_0 , \quad \tilde{S}_1 + S_1 = P_1 ,\quad  \tilde{S}_0 + S_2 = P_2 .
\] 

We set
\[
X:= S_1 + P_2,\quad Y:= S_0+ S_2-P_2 = S-(S_1+P_2)
\]
and note that they are selfadjoint, $0\le X \le \unit$, $\Gamma X \Gamma = \unit-X$ and $\Gamma Y \Gamma= -Y$. $X$ may be regarded as a ``truncation of $S$ between 0 and 1".

(2) \emph{Claim: we have}
\begin{equation}\label{eq:FF-main2}
X - R \in L^1(\K) \quad \textrm{and} \quad Y \in L^1(\K). 
\end{equation}
Since $X-R-T=-Y$ and $T\in L^1(\K)$, it suffices to show $Y\in L^1(\K)$. Moreover, since $\Gamma S_0\Gamma = P_2-S_2$, it suffices to show $S_0\in L^1(\K)$. We have
\begin{align*}
0\ge& S_0=  P_0(R+T)P_0 = P_0RP_0 + P_0TP_0,
\end{align*}
where the first term on the RHS is positive and the second one therefore negative; in other words
\[
 0 \le P_0RP_0 \le -P_0 T P_0 \in L^1(\K),
\]
so $P_0RP_0$ being dominated by a trace-class element is trace-class, too, and so is $S_0= P_0 R P_0 + P_0 T P_0$.

Since $S_0$ and $S_2-P_2$ are selfadjoint and trace-class, they can be assumed to be diagonal in a certain fixed basis, and thus also $S_2$. Write $S_0'$ for the diagonal operator which coincides with $S_0$ on the spectral subspaces of spectrum $\subset [-1/2,0]$ and equals $(-1/2)\unit$ on the (finite-dimensional) orthogonal complement in $\K_0$. Analogously define $S_2'-P_2$ and $Y':= S_0' + S_2'-P_2$, which has spectrum in $[-1/2,1/2]$. Clearly they are again trace-class and $X-Y'$ gives rise to a quasi-free state on $\CAR(\K,\Gamma)$.

(3) \emph{Claim:
\begin{equation}\label{eq:FF-main3}
\phi_{X-Y'} \simeq \phi_{X} \simeq \phi_R.
\end{equation}
In particular, $\phi_{X-Y'}$ is a state and normal with respect to the representation defined by $\phi_R$ according to Lemma \ref{lem:FF-qe-ql}.}
In order to show the quasi-equivalence of the quasi-free states, we just have to check Araki's criterion in the present case. We have
\[
 \|(X-Y')^{1/2}-X^{1/2}\|_2^2 \le \| (X- Y') - X \|_1 = \|Y'\|_1 < \infty,
\]
using (2) in the last step and the Powers-Størmer inequality in the first step, which says
\[
\|x^{1/2}-y^{1/2}\|_2^2 \le \|x-y\|_1,
\]
for all $x,y\in B(\K)_+$, \cite[Lem.4.1]{PS70}. So $\phi_{X-Y'} \simeq \phi_{X}$. A similar reasoning shows $\phi_{X}\simeq \phi_R$.

(4) \emph{Eigenspaces of $Y$.} Write $P_{02}:= P_0+P_2$ and consider the orthonormal basis $(e_j)_{j\in J}$ of $P_{2}\K$ of eigenvectors of $S_2-P_2$, with eigenvalues $\lambda_j>0$. Then $\{e_j,\Gamma e_j:\; j\in J\}$ forms an orthonormal basis for $P_{02}\K$ of eigenvectors of $Y$ with eigenvalues $\pm \lambda_j$, respectively. Clearly they are also eigenvectors of the truncation $Y'$. We remark that $P_0$ and $P_2$ are basis projections on $P_{02}\K$ in the sense of \cite[Sec.2]{Ara}.

(5) \emph{Reduction to subalgebras}. We have to show that $\phi_{X+Y}$ is bounded and that it is normal in $\pi_{\phi_R}$ or, equivalently according to (3), in $\pi_{\phi_{X-Y'}}$. To this end, it suffices to show that there is a constant $c>0$ such that
\begin{equation*}
|\phi_{X+Y}(x)| \le c \cdot \phi_{X-Y'}(x), \quad x\in \CAR(\K,\Gamma)_+ .
\end{equation*}
Notice that the  *-algebra
\[
B:= \salg\{ F(f) : f\in \{e_j,\Gamma e_j:j\in J\} \cup \K_1 \}
\]
is dense in $\CAR(\K,\Gamma)$, so $B_+\subset \CAR(\K,\Gamma)_+$ is dense, and it suffices to show
\begin{equation}\label{eq:FF-Araki-normality}
|\phi_{X+Y}(b)| \le c\cdot \phi_{X-Y'}(b), \quad b\in B_+ .
\end{equation}
We define the  *-subalgebra
\[
A:= \salg \{F(e_j)^*F(e_j), F(f) : j\in J, f\in \K_1 \}.
\]
From the canonical anticommutation relations one deduces that $A$ is isomorphic to the algebraic tensor product 
\[
\Big(\bigodot_{j\in\N} \salg\{F(e_j)^*F(e_j)\}\Big) \odot \Big(\salg\{F(f):f\in\K_1\}\Big)
\simeq (\C^2)^{\odot \N} \odot \salg\{F(f):f\in\K_1\},
\]
while $B$ as a vector space is isomorphic to the tensor product of all $\salg\{F(e_j)\}\simeq \Mat_2(\C)$ and $\salg\{F(f):f\in\K_1\}$, and $\phi_{X+Y}\restriction_B$ and $\phi_{X-Y'}\restriction_B$ split as product functionals according to this tensor product decomposition of $B$. 

We would like to construct a projection from $B$ onto the subspace $A\subset B$ compatible with this factorization. Fix the orthogonal projections $p_{j,0}:= F(e_j)^*F(e_j)$ and $p_{j,1}:= F(e_j)F(e_j)^*=\unit-p_{j,0}$ in $\salg\{F(e_j)^*F(e_j)\}$. Given any $b\in B$, we can always write it as a finite sum of elements of the form $b_1 b_2 \cdots  b_n b'$ owing to the tensor product property, with some finite $n$ and with $b_j\in \salg\{F(e_j)\}$ and $b'\in \salg\{F(f):f\in\K_1\}$. For all $n\in\N$, we then define
\[
\eta(b_1b_2 \cdots  b_n b') := \sum_{k \in \{0,1\}^n} p_{1,k_1}\cdots p_{n,k_n} (b_1 b_2 \cdots b_n b') p_{n,k_n}^*\cdots p_{1,k_1}^*.
\]
It is obvious from this definition that $\eta$ extends to a completely positive map on $B$, that all $p_{j,k}$ mutually commute, and
\[
\eta(b_1 b_2 b_3 \cdots b_n b') = \prod_{j=1}^n (p_{j,0} b_j p_{j,0} + p_{j,1} b_j p_{j,1}) \cdot b'= \eta(b_1)\eta(b_2) \eta(b_3) \cdots  \eta(b_n) \eta(b'),
\]
so $\eta$ factorizes according to the tensor product decomposition, and on every single component it is a projection, thus $\eta^2=\eta$ on all of $B$ with $\eta(B)=A$. Now, letting for a moment $\phi$ stand for either of the two $\phi_{X+Y}\restriction_B$ or $\phi_{X-Y'}\restriction_B$, we have
\begin{align*}
\phi\big(\eta\big(  \mu_1 F(e_j)^*F(e_j) &+ \mu_2 F(e_j) +\mu_3 F(e_j)^* +\mu_4 F(e_j)F(e_j)^* \big)\big)\\
=& \phi \big(\mu_1 F(e_j)^*F(e_j) +\mu_4 F(e_j)F(e_j)^* \big)\\
=&\phi \big( \mu_1 F(e_j)^*F(e_j) +\mu_2 F(e_j) +\mu_3 F(e_j)^* +\mu_4 F(e_j)F(e_j)^* \big),
\end{align*}
for all $\mu_i\in\C$ and $j\in\N$, and, since $\phi$ splits as a product functional,
\begin{align*}
\phi\eta(b_1 b_2 \cdots b_n b') =& \phi(\eta(b_1)\eta(b_2) \cdots \eta(b_n)  \eta(b'))
= \phi\eta(b_1) \phi\eta(b_2) \cdots  \phi\eta(b_n) \phi\eta(b')\\
=&\phi(b_1) \phi(b_2) \cdots \phi(b_n) \phi(b') =\phi(b_1 b_2 \cdots b_n b') . 
\end{align*}
Therefore $\phi\eta=\phi$ and we can write $b= \eta(b) + k$ with $k:=b-\eta(b) \in\ker(\phi)= \ker(\phi_{X+Y}) \cap \ker(\phi_{X-Y'})$. Now since $\phi_{X+Y}\circ \eta = \phi_{X+Y}$ and $\phi_{X-Y'}\circ \eta = \phi_{X-Y'}$ on $B$ and $\eta$ is positive, it suffices to check boundedness and \eqref{eq:FF-Araki-normality} for $x\in A_{+}$.

Finally, write $A=A_{02}\odot A_1$, with the commutative *-algebra $A_{02}:= \salg\{ F(e_j)^* F(e_j) :j\in J \}$, so
\[
\phi_{X+Y}\restriction_A = \phi_{P_{02}(X+Y)}\restriction_{A_{02}} \otimes  \phi_{P_1(X+Y)}\restriction_{A_1}, \quad 
\phi_{X-Y'}\restriction_A = \phi_{P_{02}(X-Y')}\restriction_{A_{02}} \otimes  \phi_{P_1(X-Y')}\restriction_{A_1}.
\]
Since $P_1(X+Y)= P_1(X-Y')$, which gives rise to a state, $\phi_{X+Y}\restriction_A$ and $\phi_{X-Y'}\restriction_A$ coincide on the second factor $A_1$. Thus we just have to check that $\phi_{P_{02}(X+Y)}\restriction_{A_{02}}$ is bounded and satisfies \eqref{eq:FF-Araki-normality} for $b\in A_{02,+}$, namely:

(6) \emph{Claim: there is a constant $c>0$ such that}
\begin{equation}\label{eq:FF-main4}
|\phi_{P_{02}(X+Y)}(a)| \le c |\phi_{P_{02}(X-Y')}(a)|, \quad a\in A_{02,+}.
\end{equation}
We prove this on monomials of the form
\[
a= p_{k_1,0}\cdots p_{k_n,0} p_{k_{n+1},1}\cdots p_{k_{n+m},1},
\]
which implies that it holds for arbitrary elements in $A_{02,+}$, which can always be written as linear combinations of such monomials with positive coefficients due to commutativity of $A_{02}$.
  
Let $J_2(a)=\{k_1,...,k_n\}$ and $J_0(a)=\{k_{n+1},...,k_{n+m}\}$ be the set of indices $j$ occurring in the product expansion of $a$; let $J_{0}'(a)\subset J_{0}(a)$ and $J_{2}'(a)\subset J_{2}(a)$ be the subsets of those $j$ with $0<\lambda_j<1/2$. Moreover, independently of $a$, define the constants $c_+:= \prod_{j\in J:\lambda_j\ge 1/2} (1+\lambda_j)^2$ and $c_-:= \prod_{j\in J:\lambda_j\ge 1/2} 2^2$, which are finite since the products contain only finitely many factors (owing to the trace-class condition on $Y$). With these premises, we get
\begin{align*}
|\phi_{X+Y}(a)| 
=& \Big|\prod_{j\in J_0(a)}  (-\lambda_j) \cdot \prod_{j\in J_2(a)} (1+\lambda_j) \Big|
\le c_+ \prod_{j\in J_0'(a)}  \lambda_j \cdot \prod_{j\in J_2'(a)} (1+\lambda_j)\\
\le& c_+ \prod_{j\in J_0'(a)}  \lambda_j \cdot \rme^{2\tr(|Y|)} \prod_{j\in J_2'(a)} (1-\lambda_j)
\le \rme^{2\tr(|Y|)} c_+ c_- \phi_{X-Y'}(a).
\end{align*}
The one but last inequality is true because $(1+\lambda)\le \rme^{4\lambda} (1-\lambda)$, for all $0\le\lambda \le 1/2$, and $4\sum_{j\in J} \lambda_j = 2\tr (|Y|) < \infty$. Setting $c:=  \rme^{2\tr(|Y|)} c_+ c_-$, claim (6) is proved.

(7) \emph{Conclusion.} According to (3), $\phi_{X-Y' }$ and $\phi_R$ are relatively normal. On the other hand, (5) and (6) show that $\frac{1}{c}\phi_{X+Y}$ is normal with respect to $\phi_{X-Y'}$ and bounded by the latter on $B_+$, so that it extends to a bounded functional on $\CAR(\K,\Gamma)$ which is normal with respect to $\phi_{X-Y'}$. Thus $\phi_S= \phi_{X+Y}$ is normal in the GNS representation of $\phi_R$.
\end{proof}

\section{Some technical proofs}\label{sec:proofs}

In the following proofs, we keep the notation used in the corresponding theorems and sections.

\begin{proofof}[Proposition \ref{prop:gen-alpha-inv}]
$(S_2)\Rightarrow (S_2')$. Given $x,y\in\AA_{\alpha,\phi}$, define
\[
H_{x,y}(t) := \phi(x \alpha_t(y))- F_{x,y}(t), \quad t\in \T^1.
\]
Then $H_{x,y}$ is continuous on $\T^1$, analytic on the interior of $\T^1$, and $0$ on $\R$, so $H_{x,y}=0$ owing to the edge-of-the-wedge theorem \cite[Prop.5.3.6]{BR2}. This implies in particular
\[
\phi(x\alpha_{\rmi}(y)) = F_{x,y}(\rmi) = \phi(y \gamma(x))
\]
and $|\phi(x\alpha_t(y))|=|F_{x,y}(t)| \le C_0 (1+|\Re(t)|)^{p_0}, \quad t\in\T^1$.

$(S_2') \Rightarrow (S_2)$. By definition of $\AA_{\alpha,\phi}$, for every $x,y\in\AA_{\alpha,\phi}$, the function
\begin{equation}\label{eq:gen-Fxy}
F_{x,y}(t) :=  \phi(x \alpha_t(y)), \quad t\in \T^1,
\end{equation}
satisfies the conditions in $(S_2)$, for every $x,y\in\AA_{\alpha,\phi}\cap \A(I)$. Consequently, also $(S_7)$ and $(S_8)$ hold on $\AA_{\alpha,\phi}$ so far, \cf \cite[Prop.5.3]{BG}.

In order to prove \eqref{eq:gen-S2a} on $\dom(\phi)_c$ assuming $(S_6)$, we have to work a bit more. Given arbitrary $x,y\in\A(I)$ with $\|x\|=\|y\| =1$, choose sequences $x_n,y_n\in \AA_{\alpha,\phi}\cap \A(I)$ such that 
\begin{equation}\label{eq:gen-sKMS2}
\|x_n\|,\|y_n\|\le 1, \quad x_n\ra x,   \quad y_n\ra y \quad (\sigma^*\textrm{-strongly}),\quad n\ra \infty,
\end{equation}
which is possible due to Kaplansky's density theorem if $\AA_{\alpha,\phi}\cap \A(I)\subset \A(I)$ is $\sigma$-weakly dense. 
Define \begin{equation}\label{eq:gen-Gxy}
G_{x,y}(t) := \rme^{- 4C_2(1+|I|)t^2} \phi(x \alpha_t(y)), \quad t\in \T^1,
\end{equation}
and analogously $G_{x_n,y_n}$. They are continuous and bounded on $\T^1$ owing to the exponential damping factor $(S_6)$ and local normality of $\phi$, and all $G_{x_n,y_n}$ moreover are analytic on the interior of $\T^1$ by definition of $\AA_{\alpha,\phi}$.

According to Phragmen-Lindelöf's three-line theorem \cite[Prop.5.3.5]{BR2}, all $|G_{x_n,y_n}-G_{x_m,y_m}|$, with $m,n\in\N$, attain their maximum on $\partial \T^1$. Moreover
owing to condition $(S_6)$ they tend to zero uniformly for large $t$ because on $\partial \T^1$ we have
\[
|G_{x_n,y_n}(t)| \le |\rme^{-4C_2 (1+|I|)t^2}|\, C_1 \rme^{4C_2 (1+|I|)^2} \rme^{2C_2 (1+|I|)|t|^2} \|x_n\| \|y_n\|, \quad t\in \R\cup (\R+\rmi),
\]
with $C_1,C_2>0$ as in $(S_6)$ depending only on $I$ and $\phi$, not on $x_n,y_n,t$, as is shown in detail in \cite[(3.3)]{Hil2}, or with some additional effort \cite[p.742]{BG}. For any $\eps>0$, choose $I_\eps\in\I$ symmetric around $0$ and sufficiently large such that $|I_\eps|>1$ and $C_1\rme^{4C_2 (1+|I|)^2}\rme^{-2C_2 (1+|I|)|I_\eps|^2}<\eps/4$ and such that $I\subset I_\eps$. Then, for every $t\in\T^1$:
\begin{equation}\label{eq:gen-sKMS3}
\begin{aligned}
|G_{x_n,y_n}(t) & -  G_{x_m,y_m}(t)| \le 
\max \{ \sup_{s\in \R} \rme^{-4 C_2(1+|I|) s^2} |\phi(x_n \alpha_s(y_n)) -  \phi(x_m \alpha_s(y_m))| , \\
& \sup_{s\in \R} |\rme^{-4 C_2(1+|I|) (s+\rmi)^2}|\, |\phi(\alpha_s(y_n)\gamma(x_n)) -  \phi(\alpha_s(y_m)\gamma(x_m))| \} \\
\le & \eps/2 + \rme^{4C_2(1+|I|)}\max \{\sup_{s\in I_\eps} |\phi((x_n-x_m) \alpha_s(y_n))| + \sup_{s\in I_\eps} |\phi(\alpha_{-s}(x_m)(y_n-y_m))|, \\
&\sup_{s\in I_\eps} |\phi((y_n-y_m) \alpha_{-s}\gamma(x_n))| + \sup_{s\in  I_\eps} |\phi(\alpha_s\gamma(y_m)(x_n-x_m))|\}\\
\le &  \eps/2 + \rme^{4C_2(1+|I|)}\max\Big\{ \sup_{s\in I_\eps} |\phi_{2I_\eps}|((x_n-x_m) \alpha_s(y_n)\alpha_s(y_n)^*(x_n-x_m)^*)^{1/2} \\
&\quad + \sup_{s\in I_\eps} |\phi_{2I_\eps}|((y_n-y_m)^*\alpha_{-s}(x_m)^*\alpha_{-s}(x_m)(y_n-y_m))^{1/2}, \\
&\sup_{s\in I_\eps} |\phi_{2 I_\eps}|((x_n-x_m)^* \alpha_s\gamma(y_n)^*\alpha_s\gamma(y_n)(x_n-x_m))^{1/2} \\
&\quad+ \sup_{s\in  I_\eps} |\phi_{2I_\eps}|((y_n-y_m)\alpha_{-s}\gamma(x_m)\alpha_{-s}\gamma(x_m)^*(y_n-y_m)^*)^{1/2} \Big\}\\
\le & \eps/2 +  \rme^{4C_2(1+|I|)}\max \Big\{ \|y_n \| |\phi_{2I_\eps}|((x_n-x_m)(x_n-x_m)^*)^{1/2}\\
 &\quad + \|x_m\| |\phi_{2I_\eps}|((y_n-y_m)^*(y_n-y_m))^{1/2}, \\
& \|y_n \| |\phi_{2I_\eps}|((x_n-x_m)^*(x_n-x_m))^{1/2}\\
&\quad + \|x_m\| |\phi_{2I_\eps}|((y_n-y_m)(y_n-y_m)^*)^{1/2} \Big\},
\end{aligned}
\end{equation}
where we made use of $(S_7)$ and $(S_8)$, which held on $\AA_{\alpha,\phi}\cap \A(I)$ so far. Combining this with \eqref{eq:gen-sKMS2} and the normality $(S_1)$ of $|\phi_{2 I_\eps}|$, we see that the second summand is less than $\eps/2$ for $n,m$ sufficiently large, and the whole RHS is independent of $t\in\T^1$; in other words $G_{x_n,y_n}$ converges uniformly to a bounded continuous function $G$ on $\T^1$. According to Weierstrass' convergence theorem (\eg \cite[8.4.1]{Rem}), $G$ is actually analytic on the interior of $\T^1$. Moreover, since pointwise, for every $t\in\R$, 
\[
G_{x_n,y_n}(t) \ra \rme^{- 4C_2(1+|I|)t^2} \phi(x \alpha_t(y)),\quad G_{x_n,y_n}(t+\rmi) \ra \rme^{- 4C_2(1+|I|)(t+\rmi)^2} \phi(\alpha_t(y)\gamma(x)),
\]
as $n\ra\infty$, and these values define a continuous function on $\partial \T^1$, it coincides with $G$ there. Hence the function $F_{x,y}: t\in\T^1 \mapsto  \rme^{4 C_2(1+|I|) t^2} G(t)$ is analytic on the interior of $\T^1$ and satisfies
\[
F_{x,y}(t) = \phi(x\alpha_t(y)), \quad F_{x,y}(t+\rmi) = \phi(\alpha_t(y) \gamma(x)), \quad t\in\R,
\]
which proves \eqref{eq:gen-S2a} for all $x,y\in\bigcup_{I\in\I} \A(I)$, thus in particular in $\dom(\phi)_c$. Finally, the proof of $(S_7)$ and $(S_8)$ on $\dom(\phi)$  is found in \cite[Prop.5.3]{BG}.
\end{proofof}

\begin{proofof}[Proposition \ref{th:FF-RST}]
Let
\begin{equation}\label{eq:FF-phi-integral}
\theta(f,g) = \lim_{\eps\ra 0^+} \Big( \int _{-\infty}^{-\eps} + \int_{\eps}^{\infty}  \Big)
\frac{1}{1-\rme^{-p}} \overline{\hat{f}(p)}\hat{g}(p)\rmd p, \quad f,g\in\S(\R,\C^d),
\end{equation}
and $\dom(\theta)=\S(\R,\C^d)\times \S(\R,\C^d)$. Thus $\theta$ is densely defined and sesquilinear and completely defines a quasi-free functional $(\phi,\dom(\phi))$. Since $\dom(\phi)_I\cap \CAR(\K,\Gamma)\subset \CAR(\K_I,\Gamma)$ is norm-dense, for every $I\in\I$, and $\phi$ is hermitian there, $(S_0)$ is satisfied. $(S_2)$ including its growth condition and $(S_6)$ are proved in \cite[Th.5.7\& Th.5.9]{BG} (adapted here to $\C^d$-valued functions), while $(S_3)$ is seen as follows: for every $f\in\S(\R,\R^d)$,
\begin{align*}
\phi(\unit) =& \frac{1}{\|f\|^2_2}\phi(F(f)^*F(f)+ F(f)F(f)^*) \\
=& \lim_{\eps\ra 0^+} \frac{1}{\|f\|^2_2} \sum_{j=1}^d\Big( \int_{-\infty}^{-\eps}+ \int_\eps^\infty\Big) \frac{1}{1-\rme^{-p}} \Big(\overline{\hat{f}_j(p)} \hat{f}_j(p)+\hat{f}_j(-p) \overline{\hat{f}_j(-p)}\Big) \rmd p\\
=& \lim_{\eps\ra 0^+} \frac{1}{\|f\|^2_2} \sum_{j=1}^d\Big( \int_{-\infty}^{-\eps}+ \int_\eps^\infty\Big) \Big( \frac{1}{1-\rme^{-p}} + \frac{1}{1-\rme^{p}}\Big) \overline{\hat{f}_j(p)} \hat{f}_j(p) \rmd p\\
=& \lim_{\eps\ra 0^+} \frac{1}{\|f\|^2_2} \sum_{j=1}^d\Big( \int_{-\infty}^{-\eps}+ \int_\eps^\infty\Big) \overline{\hat{f}_j(p)} \hat{f}_j(p) \rmd p\\
=& \frac{1}{\|f\|^2_2} \| \hat{f} \|_2^2 = 1
\end{align*}
by Plancherel's identity. We remark that \eqref{eq:FF-phi-integral} is the definition used in \cite[Eq.(18)]{BG} and we always refer to that form henceforth.
\end{proofof}

\begin{proofof}[Theorem \ref{th:FF-sKMS-ln}]
Let us write $P_I$ for the orthogonal projection onto the subspace $\K_I=L^2(I,\C^d) \subset L^2(\R,\C^d)=\K$, for any $I\in\I$. Note that the vacuum state restricted to the subalgebra $\CAR(\K_I,\Gamma)$ is in fact given by the quasi-free state
\[
 \phi_{P_+}\restriction_{\CAR(\K_I,\Gamma)}= \phi_{P_I P_+ P_I} = \phi_{\tilde{P}_+}.
\]
It follows from \cite[Th.5.6]{BG} that the sesquilinear form of $\phi$  in Proposition \ref{th:FF-RST} is implemented by $P_+ + T$, where $T$, defined by
\[
\langle f,T g\rangle := 2\rmi  \int_\R \int_\R \int_0^\infty \sum_{j=1}^d \overline{f_j(x)} g_j(y) (x-y) \log(1-\rme^{-p}) \cos(p(x-y)) \rmd p \rmd x \rmd y,
\]
is an unbounded operator with form domain $\S(\R,\C^d)$ and such that $P_I T P_I$ extends to a selfadjoint trace-class operator on $\K$, for every $I\in\I$. Fix one such $I$ and let $S:=\tilde{P}_+ +\tilde{T}$ with $\tilde{P}_+:= P_IP_+P_I$ and $\tilde{T}$ the closure of $P_I T P_I$; these are all bounded selfadjoint operators on $\K$ that will be naturally identified with selfadjoint operators on $\K_I$, and
\[
\tilde{P}_+ + \Gamma\tilde{P}_+\Gamma = P_I (P_+ + \Gamma P_+ \Gamma) P_I = P_I \unit P_I = P_I, \quad 0\le \tilde{P}_+ = P_I P_+ P_I \le P_I,
\]
\[
 \tilde{T} + \Gamma\tilde{T}\Gamma = P_I(T+\Gamma T \Gamma )P_I = P_I 0 P_I =0, \quad \tilde{T}\in L^1(\K_I),
\]
as follows easily from the definition of $P_+$ in \eqref{eq:FF-P+} and of $T=-\Gamma T\Gamma$ in \cite[Th.5.6(i)]{BG}, using the obvious fact that $P_I$ commutes with the complex conjugation operator $\Gamma$. Then we infer from Theorem \ref{prop:FF-comm} that $\phi_S=\phi_{\tilde{P}_++\tilde{T}}$ is bounded and normal in the GNS representation of $\phi_{\tilde{P}_+}$.

Letting $\phi_\F$ be defined through
\[
 \dom(\phi_\F) :=  \salg\{\pi_{\phi_{P_+}} (F(f)): f\in\S(\R,\R^d)\} \subset \pi_{\phi_{P_+}} (\dom(\phi))
\]
and the condition $\phi_\F\circ\pi_{\phi_{P_+}} = \phi$ with $\phi$ from Proposition \ref{th:FF-RST}, the algebraic statements and growth conditions in $(S_0)$, $(S_2)$, $(S_3)$ follow from the corresponding ones of $\phi$ by construction, with respect to the induced welldefined grading and translation automorphisms
\[
\pi_{\phi_{P_+}}(F(f)) \mapsto  \pi_{\phi_{P_+}}\circ\gamma(F(f)), \quad
\pi_{\phi_{P_+}}(F(f)) \mapsto  \pi_{\phi_{P_+}}\circ \alpha_t(F(f)), \quad f\in\K^\Gamma,
\]
denoted again by $\gamma$ and $\alpha_t$ for simplicity. 

Concerning local normality and boundedness, we have just seen that $\phi_{\tilde{P}_++\tilde{T}}$ defines a bounded normal functional $\phi_{\F,I}$ on the C*-algebra $\pi_{\phi_{\tilde{P}_+}}(\CAR(\K_I,\Gamma))$ through $\phi_{\F,I}\circ\pi_{\phi_{\tilde{P}_+}} = \phi_{\tilde{P}_++\tilde{T}}$, coinciding with $\phi_\F$ on $\dom(\phi_\F)\cap \F(I)$. Thus $\phi_{\F,I}$ extends to a unique bounded normal functional on the $\sigma$-weak closure $\F(I) = \{\pi_{\phi_{P_+}} (F(f)): f\in\K_I^\Gamma \}''$. $(S_6)$ is shown in \cite[Th.5.9]{BG}, and it holds on $\bigcup_{I\in\I}\F(I)$ owing to local normality and boundedness $(S_1)$.

Concerning uniqueness, consider the Clifford C*-algebra $A$ generated by all $F(f)$, $f\in\K^\Gamma$, which is a subalgebra of $\CAR(\K,\Gamma)$. Let
\[
\S_0(\R,\R^d) := \spa \{ \int_\R \hat{\beta}(t) f(\cdot - t) \rmd t : \beta \in \Cci_c(\R), f\in \S(\R,\R^d)\}
\subset \S(\R,\R^d),
\]
which can be checked to be again dense in $\K^\Gamma$. We follow the lines of \cite{RST}, but instead of  the KMS condition we consider the sKMS condition. Notice that the C*-dynamical system $(A,\alpha)$ is independent of the grading, so, as explained in \cite[Prop.1]{RST}, $t\in\R\mapsto \alpha_t(F(f))$ extends uniquely to an analytic function on $\C$, for every $f\in\S_0(\R,\R^d)$, and the monomials in $F(f)$, $f\in\S_0(\R,\R^d)$, lie all in the algebra $A_\alpha$ of analytic elements. Moreover, there are two commuting symmetric operators $U,\rmi V$ on the common invariant domain $\S_0(\R,\R^d)$ satisfying the conditions in \cite[Prop.2,3,4]{RST}, with $\rmi V>0$ and $U>\unit$ as quadratic forms and preserving $\S_0(\R,\R^d)$, such that $\alpha_{\rmi}(F(f))= F(Uf)+\rmi F(Vf)$, for all $f\in\S_0(\R,\R^d)$, \cf \cite[(2.12)\&(2.13)]{RST}.

Suppose there exists a functional $\psi$ on $A$ satisfying $(S_0),(S_2),(S_3)$ and $\dom(\psi)=\salg \{ F(f): f\in \S(\R,\R^d)\}$. We have to show that $\psi=\phi\restriction_{\dom(\psi)}$, which then yields the uniqueness of $\phi_\F$. Since $\psi$ is defined on monomials (by assumption), we obtain an (unbounded) hermitian sesquilinear form $\theta$ on $\S_0(\R,\R^d)\times \S_0(\R,\R^d) \subset \dom(\theta)\subset \K\times \K$ by $\theta(f,g)= \psi(F(f)^* F(g))$. Using hermiticity, the canonical anticommutation relation \eqref{eq:FF-CAR} and normalization $\psi(\unit)=1$, such a $\theta$ must satisfy
\[
\overline{\theta(g,f)}=  \theta(f,g) = -\theta(g,f)+\langle f,g \rangle, \quad f,g\in\S_0(\R,\R^d).
\]
It is useful to introduce the antisymmetric real bilinear form $\eta:= \rmi (2 \theta -\langle\cdot,\cdot\rangle)$. Let us choose
\[
A_{\alpha,\psi}:= \salg \{ F(f): f\in \S_0(\R,\R^d)\} \subset A_\alpha \cap \dom(\psi)
\]
since the functions $t\mapsto \psi(x\alpha_t(y))$ are analytic on $\T^1$, for every $x,y\in A_{\alpha,\psi}$, as shown in \cite[Th.5.6(ii)]{BG}.

Then according to Proposition \ref{prop:gen-alpha-inv} condition $(S_2')$ holds on $A_{\alpha,\psi}$:
\[
\theta(f,(U+\rmi V)g) = - \theta( g, f), \quad f,g\in\S_0(\R,\R^d),
\]
in other words
\[
\langle f,Ug \rangle - \rmi \eta(f,Ug) + \rmi \langle f,Vg \rangle + \eta(f,Vg)  
=  - \langle f, g \rangle - \rmi \eta(f, g), \quad f,g\in\S_0(\R,\R^d) .
\]
Since $f,g,\eta$ are real-valued, as in \cite[Lem.2]{RST} we may split the preceding equality into real and imaginary part, thereby obtaining
\[
 \eta(f,(U-\unit)g) = \langle f, V g \rangle,
\]
so $\eta(f,g)= \langle f,V(U-\unit)^{-1}g\rangle$ on $\S_0(\R,\R^d)\times\S_0(\R,\R^d)$. 
Applying \cite[(2.18),(2.19)]{RST} for $U$ and $V$, we see that $\theta(f, g)=\frac12(\langle f,g\rangle -\rmi \eta(f,g))$ is given by
\[
\lim_{\eps\ra 0^+} \frac12 \Big( \int _{-\infty}^{-\eps} + \int _{\eps}^{\infty}  \Big)
\Big(1+ \frac{\rme^{-p}-\rme^{p}}{2-\rme^p-\rme^{-p}}\Big) \overline{\hat{f}(p)}\hat{g}(p)\rmd p,
\]
whose continuation to $\S(\R,\R^d)\times\S(\R,\R^d)$ is \eqref{eq:FF-phi-integral} (unique owing to continuity of Fourier transformation and the denseness of $\S_0(\R,\R^d)$). The requirement of quasi-freeness then determines $\psi$ completely, namely $\psi=\phi\restriction_{\dom(\psi)}$.
\end{proofof}


\begin{proofof}[Lemma \ref{lem:FF-delta}]
The proof follows closely the local construction in \cite[Sec.5]{CHKL} making use of  the general theory in \cite[Sec.2]{CHKL}, although here we are working with the supersymmetric free field net and with $\alpha$ representing translation on $\R$ instead of rotation on $S^1$; we refer the reader to those places for further details.

The super-Virasoro net  \cite[Sec.6.1]{CKL} is the net generated by the unbounded selfadjoint even and odd operators $L(f),G(f)$ on $\H$, respectively, $f\in\Cci_c(\R)$, which satisfy the following (graded) commutation relations on the common invariant core $\D_\infty\subset\H$ (\cf \cite[Lem.4.5\&4.6]{CLTW2} together with \cite[Sec.4]{CHKL} and \cite[Ex.2.10]{CHL}):
\begin{equation}\label{eq:FF-SVir}
\begin{gathered}
\ [L(f), L(g)] = -\rmi L(f'g)+ \rmi L(fg') +\rmi\frac{c}{12} \int_{\R} f'''g,\\
[L(f), G(g)] = \rmi G(f g') -\frac{\rmi}{2} G( f'g),\\
[G(f), G(g)] = 2L (fg) + \frac{c}{3} \int_{\R} f'g'.
\end{gathered}
\end{equation}
The super-Virasoro field operators $L(f), G(f)$ are obtained from the supersymmetric free field by the super-Sugawara construction \cite[(6.4)]{CHL} (with $l+g=1$ and $J^{\pi^\F}=0$ there, due to the U(1)-current or free field commutation relations). They satisfy the following (graded) commutation relations:
\begin{equation}\label{eq:FF-SVirFF}
\begin{gathered}
\ [J(f),L(g)] = -\rmi J(f'g), \quad [F(f),L(g)]= -\rmi F(f'g)-\frac{\rmi}{2} F(fg'),\\
[J(f),G(g)] = -\rmi F(f'g), \quad [F(f),G(g)]= J(fg),
\end{gathered}
\end{equation}
for $f\in\Cci_c(\R,\R^d)$ and $g\in\Cci_c(\R)$. These relations are again to be understood on the common invariant core $\D_\infty$. This shows that the free field net $\A$ on $\H$ contains the super-Virasoro net $\A_{\SVir,c}$ with central charge $c=\frac32 d$ as subnet, \cf \cite[Def.2.11\&Prop.6.2]{CHL}; they are covariant with respect to the same projective representation of $\Diff^{(\infty)}$, so $\A$ is superconformal.

Let $\varphi_I\in\Cci_c(\R)$ denote an arbitrary but fixed function with $\varphi_I\restriction _I=1$, and set $Q_I:= G(\varphi_I)$, which defines the superderivation $\delta_{Q_I}$ as prescribed in \eqref{eq:FF-dom-QI}. Then according to \eqref{eq:FF-SVirFF} $\delta_{Q_I}$ satisfies \eqref{eq:FF-delta-formal} on $\D_\infty$. Applying then Borel functional calculus, almost literally as in the proof of \cite[Prop.6.4]{CHL} but with $Q_I$ instead of $Q$ and $J(f)$ U(1)-currents instead of $G$-currents, we obtain \eqref{eq:FF-delta-vNa} on $\AA_0\cap\A(I)\subset \dom(\delta_{Q_I})$, which was to be proved. In particular, in analogy to \cite[Prop.5.3]{CHKL}, $\delta_{Q_I}$ does not depend on the actual choice of $\varphi_I\in\Cci_c(\R)$ as long as $\varphi_I\restriction _I=1$.

According to the general theory of superderivations implemented as graded commutators in \cite[Prop.2.1]{CHKL}, it follows that $\delta_{Q_I}$ is ($\sigma$-weak)-($\sigma$-weakly) closed and satisfies property $(ii)$.
\end{proofof}

\begin{proofof}[Lemma \ref{lem:FF-xh}]
The generator $P$ of the translation group is positive by assumption and preserves $\D_\infty\subset\H$, the above-mentioned common invariant core for all the field operators $F(f),J(f),G(g),L(g)$, with $f\in\Cci_c(I,\R^d)$ and $g\in\Cci_c(I)$. On this core, we have
\[
 L(\varphi_I^2) X(f) \xi - X(f)L(\varphi_I^2) \xi
= \rmi  X(f')\xi
= P X(f) \xi  - X(f)P \xi, \quad \xi\in\D_\infty,
\]
according to \eqref{eq:FF-SVirFF} and  the definition of $P$, with $X$ standing for $J,F$. Passing then to resolvents and applying Borel functional calculus as in the proof of \cite[Prop.6.4]{CHL} with $Q$ replaced by $L(\varphi_I^2)-P$, we find that
\[
 x (L(\varphi_I^2)-P) \subset (L(\varphi_I^2)-P) x + 0, 
\]
for all $x\in\AA_0\cap\A(I)$.

The rest of the proof goes along the lines of \cite[Prop.5.6]{CHKL}, except that we are now over $\R$ instead of $S^1$, and $\alpha$ implements translation instead of rotation, namely: given $x\in\AA_0\cap\A(\frac12 I)$ and $h\in\Cci_c(\frac12 I)$ and $Q_I:= G(\varphi_I)$ as in the proof of the preceding lemma, we see that 
\[
Q_I\alpha_t(x)\xi = \alpha_t(Q_{I-t} x)\xi = \alpha_t(\gamma(x) Q_{I-t}) \xi + \alpha_t(\delta_{Q_{I-t}}(x)) \xi
= \gamma\alpha_t(x) Q_{I} \xi + \alpha_t(\delta_{Q_I}(x)) \xi,    
\]
for all $t\in\frac12 I$.
Thus $x_h\in \dom(\delta_{Q_I})$, with $\delta_{Q_I}(x_h)=(\delta_{Q_I}(x))_h$ using simply the definition of $x_h$, and analogously as above, we obtain $x_h (L(\varphi_I^2)-P) \subset (L(\varphi_I^2)-P) x_h$. But since
\[
 x_h= \int_\R \alpha_t(x) h(t) \rmd t
=\int_\R \rme^{\rmi t P} x \rme^{-\rmi t P} h(t) \rmd t,
\]
$x_h$ lies in the domain of the commutator with $P$, whose closure is given by $\delta_0(x_h)=\rmi x_{h'}$, \cf \cite[Lem.2.11]{CHKL} and \cite[Prop.3.2.55]{BR2}.

Moreover, since $x_h\xi, \delta_{Q_I}(x_h) \xi\in\D_\infty$, for all $\xi\in\D_\infty$, we may compute as follows:
\begin{align*}
G(\varphi_I)\delta_{Q_I}(x_h)\xi 
= & G(\varphi_I)^2 x_h\xi- G(\varphi_I)\gamma(x_h)G(\varphi_I)\xi  \\
= & G(\varphi_I)^2 x_h\xi 
-\delta_{Q_I}(\gamma(x_h))G(\varphi_I)\xi
-x_hG(\varphi_I)^2\xi \\
= & L(\varphi_I^2)x_h\xi -x_h L(\varphi_I^2)\xi
-\delta_{Q_I}(\gamma(x_h))G(\varphi_I) \xi \\ 
= & (L(\varphi_I^2)-P) x_h\xi -x_h (L(\varphi_I^2)-P)\xi +P x_h \xi - x_h P \xi
-\delta_{Q_I}(\gamma(x_h))G(\varphi_I) \xi \\ 
= & P x_h\xi -x_h P\xi 
-\delta_{Q_I}(\gamma(x_h))G(\varphi_I) \xi \\
= & \rmi x_{h'}\xi -\delta_{Q_I}(\gamma(x_h))G(\varphi_I)\xi \\
= & \rmi x_{h'}\xi + \gamma(\delta_{Q_I}(x_h))G(\varphi_I)\xi, 
\end{align*}
thus $x_h\in\dom(\delta_{Q_I}^2)$ with $\delta_{Q_I}^2(x_h) = \delta_0(x_h) = \rmi x_{h'}$. Since $h'\in\Cci_c(\frac12 I)$ again, one continues recursively to obtain the statement of the lemma.

\end{proofof}

\begin{proofof}[Theorem \ref{th:FF-sKMS1} (remaining part)]
With assumptions and notation as above and in the first part of the proof of Theorem \ref{th:FF-sKMS1}, we have to show
\[
\psi_I\circ \delta_I (x) = 0, \quad x \in B(I):= \salg\{y, \delta(y): y\in\AA_0\cap\A(I)\}.
\]

We follow the proof of \cite[Th.5.8]{BG}. Although the Weyl operators are not treated there explicitly, the strong regularity of the quasifree state $\phi_\B$ together with the discussion in \cite[p.705-706 \& (10)]{BG} shows that the functional $\psi_I$ (as introduced there with the symbol $\phi$) is well-defined on
\[
 \salg\{ F(f)(J(f)+\rmi)^{-1}, J(f), (J(f)+\rmi)^{-1}, W(f): f\in\Cci_c(I,\R^d)\},
\]
so in particular on $B(I)$, which generalizes \cite[Lem.8.2]{BG}. With $f_i\in\Cci_c(I,\R^d)$, we calculate (writing $R(f)$ for $(J(f)+\rmi)^{-1}$), for all integers $0\le l\le m\le n$:
\begin{align*}
\psi_I\circ\delta_I & \Big(F(f_1) \cdots F(f_l) R(f_1)\cdots R(f_l) J(f_{l+1})\cdots J(f_{m})
W(f_{m+1})\cdots W(f_n)\Big) \\
=& (-1)^l(-\rmi)^{m-l} \int_0^\infty \cdots\int_0^\infty \frac{\partial^{m-l}}{\partial t_{l+1}\cdots\partial t_m}\\
&\quad\psi_I\circ\delta_I\Big(F(f_1)\cdots F(f_l) W(t_1f_1) \cdots W(t_n f_n)\Big) \rme^{-t_1-...-t_l} \rmd t_1 \cdots \rmd t_l \restriction^{t_{l+1}...t_m =0}_{t_{m+1}...t_n=1} \\
=& \sum_{i=1}^l (-1)^{i-1}\psi_I\Big(F(f_1)\cdots\v i.\cdots F(f_l)\Big) 
(-1)^l(-\rmi)^{m-l} \int_0^\infty\cdots\int_0^\infty  -\rmi \frac{\partial^{m-l+1}}{\partial t_0 \partial t_{l+1}\cdots\partial t_m}\\
&\quad
\psi_I\Big( W(t_0f_i) W(t_1 f_1) \cdots W(t_n f_n)\Big) \rme^{-t_1-...-t_l} \rmd t_1 \cdots \rmd t_l \restriction^{t_0,t_{l+1}...t_m =0}_{t_{m+1}...t_n=1} \\
&+ (-1)^l\sum_{k=1}^n\psi_I\Big(F(f_1)\cdots F(f_l)F(f_k')\Big)
(-1)^l(-\rmi)^{m-l} \int_0^\infty\cdots\int_0^\infty  \frac{\partial^{m-l}}{\partial t_{l+1}\cdots\partial t_m}\\
&\quad
\psi_I\Big( W(t_0f_i) W(t_1 f_1) \cdots W(t_n f_n)\Big) (-t_k)\rme^{-t_1-...-t_l} \rmd t_1 \cdots \rmd t_l \restriction^{t_0,t_{l+1}...t_m =0}_{t_{m+1}...t_n=1}.
\end{align*}
By means of derivatives and Laplace transforms, we reduced everything on the bosonic free field part to products of Weyl unitaries because for the latter ones we have the following handy expression (from \cite[p.706\&(15)]{BG}):
\[
\psi_I(W(f_1)\cdots W(f_n)) = \exp \Big( -\sum_{j<k} \rmi\theta(f_j,f_k') - \frac12 \sum_{j=1}^n \rmi\theta(f_j,f_j') \Big),
\]
where $\theta$ is the two-point function for $\phi_\F$ defined in \eqref{eq:FF-RST-2pt}.
The quasifreeness of $\phi_\F$ implies that the above expression can be nonzero only if $l$ is odd, which we shall assume from now on. It implies also that products of $(l+1)$ fields can be split into $2$ and $(l-1)$ in the usual manner, so that we obtain for the above expression
\begin{align*}
\psi_I\circ\delta_I\Big(F(f_1) & \cdots F(f_l) R(f_1)\cdots R(f_l) J(f_{l+1})\cdots J(f_{m})
W(f_{m+1})\cdots W(f_n)\Big) \\
=& \sum_{i=1}^l (-1)^{i}\psi_I\Big(F(f_1)\cdots\v i.\cdots F(f_l)\Big) 
(-\rmi)^{m-l} \int_0^\infty\cdots\int_0^\infty  \frac{\partial^{m-l}}{\partial t_{l+1}\cdots\partial t_m}\\
&\quad
\Big(-\rmi \frac{\partial}{\partial t_0} - \sum_{k=1}^n (-t_k) \psi_I(F(f_i)F(f_k')) \Big)\\
&\quad
\psi\Big( W(t_0f_i) W(t_1 f_1) \cdots W(t_n f_n)\Big) \rme^{-t_1-...-t_l} \rmd t_1 \cdots \rmd t_l \restriction^{t_0,t_{l+1}...t_m =0}_{t_{m+1}...t_n=1} \\
=& \sum_{i=1}^l (-1)^{i}\psi_I\Big(F(f_1)\cdots\v i.\cdots F(f_l)\Big) 
(-\rmi)^{m-l} \int_0^\infty\cdots\int_0^\infty   \frac{\partial^{m-l}}{\partial t_{l+1}\cdots\partial t_m}\\
&\quad
\Big(-\rmi \sum_{j=1}^n t_j (-\rmi)\theta(f_i,f_j') - \sum_{k=1}^n (-t_k) \theta(f_i,f_k') \Big)\\
&\quad
\psi_I\Big( W(t_0f_i) W(t_1 f_1) \cdots W(t_n f_n)\Big) \rme^{-t_1-...-t_l} \rmd t_1 \cdots \rmd t_l \restriction^{t_0,t_{l+1}...t_m =0}_{t_{m+1}...t_n=1}\\
=& 0
\end{align*}
Thus $\psi_I\circ\delta_I=0$ on $B(I)$, for every $I\in\I$.
\end{proofof}

\bigskip

{\footnotesize 

\noindent\textbf{Acknowledgements.}
I would like to thank Paolo Camassa, Sebastiano Carpi, and Roberto Longo for several helpful and pleasant discussions, for comments and for pointing out problems.
}


\bigskip


\begin{thebibliography}{CLTW11b}

\bibitem[Ara70]{Ara}
H.~Araki.
\newblock {On quasi-free states of CAR and Bogoliubov automorphisms}.
\newblock {\em Publ. Res. Inst. Math. Sci.} 6, 385-442, 1970.

\bibitem[Ara71]{Ara2}
H.~Araki.
\newblock {On quasi-free states of the canonical commutation relations. II}. 
\newblock {\em Publ. Res. Inst. Math. Sci.} 7, 121–152, 1971.

\bibitem[AM71]{Ara1}
H.~Araki and S.~Masafumi.
\newblock {On quasi-free states of the canonical commutation relations. I}. 
\newblock {\em Publ. Res. Inst. Math. Sci.} 7, 105–120, 1971.

\bibitem[Böc96]{Boc}
J.~Böckenhauer.
\newblock {Localized endomorphisms of the chiral Ising model}.
\newblock {\em Commun. Math. Phys.} 177, 265--304, 1996.

\bibitem[BR97]{BR2}
O.~Bratteli and D.~Robinson.
\newblock {\em Operator algebras and quantum statistical mechanics}. 
\newblock Springer, 1997.

\bibitem[BG07]{BG}
D.~Buchholz and H.~Grundling.
\newblock {Algebraic supersymmetry: a case study}.
\newblock {\em Commun. Math. Phys.} 272, 699--750, 2007.

\bibitem[BJ89]{BJ}
D.~Buchholz and P.~Junglas.
\newblock {On the existence of equilibrium states in local quantum field theory}. 
\newblock {\em Commun. Math. Phys.} 121, 255–270, 1989.

\bibitem[BL00]{BL00}
D.~Buchholz and R.~Longo.
\newblock Graded KMS functionals and the breakdown of supersymmetry.
\newblock {\em Adv. Theor. Math. Phys.}  3, 615--626, 2000.
\newblock Addendum: {\em Adv. Theor. Math. Phys.}  6, 1909--1910, 2000.

\bibitem[BMT88]{BMT}
D.~Buchholz, G.~Mack, and I.~Todorov.
\newblock {The current algebra on the circle as a germ of local field
  theories}.
\newblock {\em Nucl. Phys. B-Proc. Suppl.} 5, 20--56, 1988.

\bibitem[BSM90]{BSM}
D.~Buchholz and H.~Schulz-Mirbach.
\newblock {Haag duality in conformal quantum field theory}.
\newblock {\em Rev. Math. Phys} 2, 105--125, 1990.

\bibitem[CLTW12a]{CLTW1}
P.~Camassa, R.~Longo, Y.~Tanimoto, and M.~Weiner.
\newblock {Thermal states in conformal QFT. I}.
\newblock {\em Commun. Math. Phys.} 309, 703--735, 2012.

\bibitem[CLTW12b]{CLTW2}
P.~Camassa, R.~Longo, Y.~Tanimoto, and M.~Weiner.
\newblock {Thermal states in conformal QFT. II}.
\newblock {\em Commun. Math. Phys.} 315, 771--802, 2012.

\bibitem[Car04]{Carpi}
S.~Carpi.
\newblock {On the representation theory of Virasoro nets}.
\newblock {\em Commun. Math. Phys.} 244, 261--284, 2004.

\bibitem[CCHW13]{CCHW}
S.~Carpi, R.~Conti, R.~Hillier, and M.~Weiner.
\newblock {Representations of conformal nets, universal C*-algebras and K-theory}.
\newblock {\em Commun. Math. Phys.} 320, 275--300, 2013

\bibitem[CHKL10]{CHKL}
S.~Carpi, R.~Hillier, Y.~Kawahigashi, and R.~Longo.
\newblock {Spectral triples and the super-Virasoro algebra}.
\newblock {\em Commun. Math. Phys.} 295, 71--97, 2010.

\bibitem[CHKLX13]{CHKLX}
S.~Carpi, R.~Hillier, Y.~Kawahigashi, R.~Longo and F.~Xu.
\newblock {$N=2$ superconformal nets}.
\newblock {\em arXiv:}1207.2398v3 [math.OA], 2013.

\bibitem[CHL13]{CHL}
S.~Carpi, R.~Hillier, and R.~Longo.
\newblock {Superconformal nets and noncommutative geometry}.
\newblock {\em J. Noncomm. Geom.}, to appear.
\newblock {\em arXiv:}1304.4062v2 [math.OA], 2013.

\bibitem[CKL08]{CKL}
S.~Carpi, Y.~Kawahigashi, and R.~Longo.
\newblock {Structure and classification of superconformal nets}.
\newblock {\em Ann. Henri Poincar\'e} 9, 1069--1121, 2008.

\bibitem[Con94]{Con94}
A.~Connes.
\newblock {\em Noncommutative geometry}.
\newblock {Academic Press}, 1994.

\bibitem[FG93]{FG}
J.~Fröhlich and F.~Gabbiani.
\newblock {Operator algebras and conformal field theory}.
\newblock {\em Commun. Math. Phys.} 155, 569--640, 1993.

\bibitem[FRS92]{FRS} K.~Fredenhagen, K.H. Rehren, and B.~Schroer.
\newblock Superselection sectors with braid group statistics and exchange
   algebras II. Geometric aspects and conformal covariance.
\newblock {\em Rev. Math. Phys.(Special Issue)} 113--157, 1992.

\bibitem[Haa92]{Haag}
R.~Haag.
\newblock {\em {Local quantum physics}}.
\newblock Springer, 1992.

\bibitem[Hil14]{Hil2}
R.~Hillier.
\newblock {Local-entire cyclic cocycles for graded quantum field nets}.
\newblock {\em Lett. Math. Phys.} 104, 271-298, 2014.

\bibitem[JLW89]{JLW}
A.~Jaffe, A.~Lesniewski, and M.~Wisniowski.
\newblock {Deformations of super-KMS functionals}.
\newblock {\em Commun. Math. Phys.} 121, 527--540, 1989.

\bibitem[KT85]{KT}
V.G. Kac and I.T. Todorov.
\newblock {Superconformal current algebras and their unitary representations}.
\newblock {\em Commun. Math. Phys.} 102, 337--347, 1985.

\bibitem[KR86]{KR2}
R.V. Kadison and J.R. Ringrose.
\newblock {\em Fundamentals of the theory of operator algebras}.
\newblock Academic Press, 1986.

\bibitem[Kas89]{Kas}
D.~Kastler.
\newblock {Cyclic cocycles from graded KMS functionals}.
\newblock {\em Commun. Math. Phys.} 121, 345–350, 1989. 

\bibitem[KL05]{KL06}
Y.~Kawahigashi and R.~Longo.
\newblock {Noncommutative spectral invariants and black hole entropy}.
\newblock {\em Commun. Math. Phys.} 257, 193–225, 2005.

\bibitem[KLM01]{KLM}
Y.~Kawahigashi, R.~Longo, and M.~M{\"u}ger.
\newblock Multi-interval subfactors and modularity of representations in
   conformal field theory.
\newblock {\em Commun. Math. Phys.} 219, 631--669, 2001.

\bibitem[Lon01]{Lo01}
R.~Longo.
\newblock {Notes for a quantum index theorem}.
\newblock {\em Commun. Math. Phys.} 222, 45–96, 2001. 

\bibitem[MS59]{MS}
P.~Martin and J.~Schwinger.
\newblock {Theory of many-particle systems. I}.
\newblock {\em Phys. Rev.} 115, 1342–1373, 1959.

\bibitem[Mor10]{Mor10}
H.~Moriya.
\newblock {Supersymmetric C*-dynamical systems}.
\newblock {\em arXiv:}1001.2622 [math.OA], 2010.

\bibitem[Mor11]{Mor11}
H.~Moriya.
\newblock {On GNS representation of supersymmetric states in C*-dynamical systems}.
\newblock {\em Mathematical quantum field theory and renormalization theory}.
\newblock {\em COE Lect. Note} 30, 39-47, 2011.

\bibitem[PS70]{PS70}
R.T. Powers and E.~St{\o}rmer.
\newblock Free states of the canonical anticommutation relations.
\newblock {\em Commun. Math. Phys.} 16, 1--33, 1970.

\bibitem[Rem91]{Rem}
R.~Remmert
\newblock {\em Theory of complex functions}.
\newblock Springer, 1991.

\bibitem[RST69]{RST}
F.~Rocca, M.~Sirugue, and D.~Testard.
\newblock {On a class of equilibrium states under the Kubo-Martin-Schwinger
  boundary condition. Fermions}.
\newblock {\em Commun. Math. Phys.} 13, 317--334, 1969.

\bibitem[RST70]{RST2}
F.~Rocca, M.~Sirugue, and D.~Testard.
\newblock {On a class of equilibrium states under the Kubo-Martin-Schwinger
  boundary condition. Bosons}.
\newblock {\em Commun. Math. Phys.} 19, 119–141, 1970.

\bibitem[Sto07]{Sto06}
O.~Stoytchev.
\newblock Modular conjugation and the implementation of supersymmetry.
\newblock {\em Lett. Math. Phys.} 79, 235--249, 2007.

\bibitem[TW73]{TW}
M.~Takesaki and M.~Winnink.
\newblock Local normality in quantum statistical mechanics.
\newblock {\em Commun. Math. Phys.} 30, 129--152, 1973.

\bibitem[Tak79]{Tak}
M.~Takesaki
\newblock {\em Theory of operator algebras I}.
\newblock Springer, 1979.

\bibitem[Wei05]{Wei05}
M.~Weiner.
\newblock Conformal covariance and related properties of chiral QFT.
\newblock {\em PhD Thesis. Universit\`a di Roma ``Tor Vergata''}, arxiv:math/0703336 [math.OA], 2005.

\bibitem[Xu05]{Xu1}
F.~Xu.
\newblock Strong additivity and conformal nets.
\newblock {\em Pac. J. Math.} 221, 167--199, 2005.


\end{thebibliography}
\end{document}